\documentclass[12pt]{amsart}

\usepackage{amsmath, amstext, amsfonts, amssymb, dsfont, amsthm}
\usepackage{amsrefs}
\usepackage{mathtools}
\usepackage{leftindex}
\usepackage{color}
\usepackage{graphicx}
\usepackage[colorlinks=true, allcolors=black]{hyperref}
\usepackage{float}
\allowdisplaybreaks
\usepackage{tikz-cd}
\usepackage{amsbsy}
\usepackage{amscd}
\usetikzlibrary{matrix,arrows,decorations.pathmorphing}
\usepackage{enumitem}
\usepackage{setspace}

\usepackage[english]{babel}
\usepackage[utf8x]{inputenc}
\usepackage[T1]{fontenc}

\usepackage[margin=1.0in]{geometry}

\newtheorem{thm}{Theorem}[section]
\newtheorem{lem}[thm]{Lemma}
\newtheorem{cor}[thm]{Corollary}
\newtheorem{prop}[thm]{Proposition}

\theoremstyle{definition}

\theoremstyle{definition}

\theoremstyle{definition}
\newtheorem{defn}[thm]{Definition}
\theoremstyle{definition}
\newtheorem{remark}[thm]{Remark}

\newcommand{\lo}{\ell^1(G)}
\newcommand{\LO}{L^1(G)}
\newcommand{\LOQ}{L^1(\mathbb{G})}

\newcommand{\LOQH}{L^1(\widehat{\mathbb{G}})}

\newcommand{\lt}{\ell^2(G)}

\newcommand{\LTQ}{L^2(\mathbb{G})}
\newcommand{\LTQH}{L^2(\widehat{\mathbb{G}})}

\newcommand{\LI}{L^{\infty}(G)}
\newcommand{\LIQ}{L^{\infty}(\mathbb{G})}

\newcommand{\liq}{\ell^{\infty}(\mathbb{G})}
\newcommand{\loq}{\ell^{1}(\mathbb{G})}
\newcommand{\ltq}{\ell^{2}(\mathbb{G})}

\newcommand{\bim}[1]{\mathrm{Bim}(#1^{\perp})}
\newcommand{\msig}[1]{ m_{\sigma}^r(#1)}
\newcommand{\nul}{\mathfrak{N}}

\newcommand{\LIQH}{L^{\infty}(\widehat{\mathbb{G}})}

\newcommand{\LIQHH}{L^{\infty}(\widehat{\widehat{\mathbb{G}}})}
\newcommand{\LIQHP}{L^{\infty}(\widehat{\mathbb{G}}')}

\newcommand{\LIQHOP}{L^{\infty}(\widehat{\mathbb{G}}\op)}
\newcommand{\LOQHOP}{L^1(\widehat{\mathbb{G}}\op)}

\newcommand{\BH}{\mc{B}(H)}
\newcommand{\Th}{\mc{T}(H)}

\newcommand{\BLT}{\mc{B}(L^2(G))}
\newcommand{\KLT}{\mc{K}(L^2(G))}
\newcommand{\BLTQ}{\mc{B}(L^2(\mathbb{G}))}
\newcommand{\bltq}{\mc{B}(\ell^2(\mathbb{G}))}

\newcommand{\kltq}{\mc{K}(\ell^2(\mathbb{G}))}
\newcommand{\TCQ}{\mc{T}(L^2(\mathbb{G}))}

\newcommand{\McbQl}{M_{cb}^l(L^1(\mathbb{G}))}
\newcommand{\QcbQl}{Q_{cb}^l(L^1(\mathbb{G}))}
\newcommand{\QcbQHl}{Q_{cb}^l(L^1(\wh{\mathbb{G}}))}
\newcommand{\McbQHl}{M_{cb}^l(L^1(\widehat{\mathbb{G}}))}

\newcommand{\McbQr}{M_{cb}^r(L^1(\mathbb{G}))}
\newcommand{\QcbQr}{Q_{cb}^r(L^1(\mathbb{G}))}

\newcommand{\tl}{\Theta^\ell}
\newcommand{\tlf}{\wh{\tl}\op(\wh{f_{\iota}})}
\newcommand{\tll}[1]{\wh{\tl}\op(\wh{#1})}
\newcommand{\tf}{\wh{\tl}\op(\wh{f}\,)}
\newcommand{\trf}{\Theta^r(f)}
\newcommand{\m}{\mathsf{M}}
\newcommand{\n}{\mathsf{N}}
\newcommand{\cb}[3]{\mathcal{CB}_{#1}^{\sigma, #2}(#3)}
\newcommand{\ext}{\textnormal{ext}}
\newcommand{\crs}[2]{{#1} \leftindex_\al {\ltimes}  {#2}}
\newcommand{\crsred}[2]{{#1} \leftindex_\al {\ltimes}_r {#2}}
\newcommand{\phih}{\wh{\vphi}}

\newcommand{\uih}[2]{\wh{u}_{#1}^{#2}}
\newcommand{\uis}[2]{(\wh{u}_{#1}^{#2})^*}
\newcommand{\fcrs}[2]{{#1}  \ltimes_{\mathcal{F}} {#2}}
\newcommand{\ten}{\otimes}
\newcommand{\oten}{\overline{\otimes}}
\newcommand{\pten}{\widehat{\otimes}}

\newcommand{\mten}{\otimes_{\textnormal{min}}}

\newcommand{\ften}{\oten_{\mc{F}}}

\newcommand{\vphi}{\varphi}

\newcommand{\G}{\mathbb{G}}

\newcommand{\C}{\mathbb{C}}
\newcommand{\N}{\mathbb{N}}

\DeclareSymbolFont{lettersA}{U}{txmia}{m}{it}
\DeclareMathSymbol{\W}{\mathord}{lettersA}{151}

\newcommand{\al}{\alpha}
\newcommand{\be}{\beta}
\newcommand{\lm}{\lambda}

\newcommand{\Gam}{\Gamma}
\newcommand{\om}{\omega}
\newcommand{\Om}{\Omega}
\newcommand{\wtal}{\widetilde{\alpha}}

\newcommand{\Irr}{\mathrm{Irr}(\mathbb{G})}
\newcommand{\Irrd}{\mathrm{Irr}(\mathbb{\widehat{G}})}

\newcommand{\id}{\textnormal{id}}

\newcommand{\op}{^{\textnormal{op}}}
\newcommand{\pol}{\textnormal{Pol}(\G)}
\newcommand{\polh}{\textnormal{Pol}(\wh{\G})}

\newcommand{\uij}{u_{ij}^{\al}}

\providecommand{\norm}[1]{\lVert#1\rVert}

\newcommand{\h}[1]{\widehat{#1}}
\newcommand{\wh}[1]{\widehat{#1}}

\newcommand{\mc}[1]{\mathcal{#1}}

\newcommand{\la}{\langle}
\newcommand{\ra}{\rangle}
\newcommand{\lBr}{\biggl(}
\newcommand{\rBr}{\biggr)}
\newcommand{\tr}{\mathrm{tr}}
\newcommand{\rmv}[1]{}

\newcommand{\ww}{\widetilde{W}}

\begin{document}

\title[]{Fej\'{e}r representations for discrete quantum groups and applications}

\author{Jason Crann}
\email{jasoncrann@cunet.carleton.ca}
\address{School of Mathematics and Statistics, Carleton University, Ottawa, ON, Canada K1S 5B6}

\author{Soroush Kazemi}
\email{soroushkazemi@cmail.carleton.ca}
\address{School of Mathematics and Statistics, Carleton University, Ottawa, ON, Canada K1S 5B6}

\author{Matthias Neufang}
\email{mneufang@math.carleton.ca}
\address{School of Mathematics and Statistics, Carleton University, Ottawa, ON, Canada K1S 5B6}


\begin{abstract}
We prove that a discrete quantum group $\G$ has the approximation property if and only if a Fej\'{e}r-type representation holds for its $C^*$--algebraic or von Neumann algebraic crossed products. As applications, we extend several results from \cite{crannNeufang2022Fejer,anoussis2014ideals-A(G),anoussis2016idealsFourier} to the context of discrete quantum groups with the approximation property. Additionally, we provide new characterizations of invariant $\LIQH$--bimodules of $\mc{B}(\ltq)$ and invariant $C(\wh{\G})$--bimodules of $\mc{K}(\ltq)$, some of which are new in the group setting. Finally, we study Fubini crossed products of discrete quantum group actions.
\end{abstract}

\begin{spacing}{1.0}

\maketitle
\section{Introduction}
In the early 20th century, Fej\'{e}r established, under appropriate conditions, the approximation of a function by the Ces\`{a}ro sum of its Fourier series \cite{fejer1903untersuchungen}. More specifically, if $f\in L^{\infty}(\mathbb{T})$ with Fourier series $S_N(f)(t)=\sum_{n=-N}^{N}\wh{f}(n)e^{int}$, then
\begin{equation}\label{intro}
    \frac{1}{N}\sum_{n=0}^{N-1}S_n(f)=F_N*f\rightarrow f
\end{equation}
weak* (uniformly if $f \in C(\mathbb{T})$), where $F_N(t) = \frac{1}{N} \sum_{n=0}^{N-1} \sum_{k=-n}^{n} e^{ikt}$ is Fej\'{e}r's kernel. Shortly after he gave explicit examples of continuous periodic functions whose Fourier series do not converge pointwise \cite{fejer1911singularites}.

We may interpret the Ces\`{a}ro convergence \eqref{intro} through pointwise multiplication under the Fourier transform: the sequence $(\widehat{F}_N)$ forms a bounded approximate identity for the Fourier algebra $A(\mathbb{Z})$, and we have $\widehat{F}_N \cdot x \to x$ weak* for any $x \in VN(\mathbb{Z}) \cong L^{\infty}(\mathbb{T})$ (uniformly if $x \in C^*_{\lambda}(\mathbb{Z}) \cong C(\mathbb{T})$), where $\cdot$ is the canonical pointwise action of $A(\mathbb{Z})$ on $VN(\mathbb{Z})$. It follows that
$$ x = w^* - \lim_{N} \frac{1}{N} \sum_{n=0}^{N-1} \chi_{[-n,n]} \cdot x = w^* - \lim_{N} \frac{1}{N} \sum_{n=0}^{N-1} \sum_{k=-n}^{n} \tau(x \lambda(k)^*) \lambda(k) $$
provides a Fej\'{e}r representation for any $x$ in the von Neumann crossed product $VN(\mathbb{Z}) = \mathbb{Z} \ltimes \mathbb{C}$ through an explicit linear combination of translation operators, where $\tau$ is the canonical tracial state on $VN(\mathbb{Z})$. Note the appearance of a Følner sequence for $\mathbb{Z}$, linking the Ces\`{a}ro summability to amenability of $\mathbb{Z}$.

Similar Fej\'{e}r-type representations exist for non-trivial crossed products $G \ltimes \n$ where $G$ is a locally compact abelian group acting on a von Neumann algebra $\n$ (see, for instance, \cite{zeller1968produit} or \cite[§7.10]{Pedersen2018automorphism}), where the coefficients are now $\n$-valued. In that setting, the Ces\`{a}ro sum is replaced by a suitable average over a bounded approximate identity in $L^1(\widehat{G}) \cong A(G)$, again linking Fej\'{e}r representations to amenability of $G$.

If $G$ is a discrete group acting on a von Neumann algebra $\n$, then the "Fourier series"
\begin{equation}\label{intro2}
    \sum_{s \in G} E(x u(s)^*) u(s)
\end{equation}
is in general not strongly or weak* convergent to $x$ for every $x \in G \ltimes \n$, even in the case of $\mathbb{Z}$ acting trivially on $\mathbb{C}$, as mentioned above. Here, $E : G \ltimes \n \rightarrow \n$ is the canonical conditional expectation and $u(s)$ is the image of the regular representation in the crossed product. Summability properties of \eqref{intro2} and related questions concerning the Fourier analysis of $C^*$-- and von Neumann crossed products have been extensively studied. In particular, Fej\'{e}r-type representations for elements of crossed products have been considered over (weakly) amenable discrete groups \cite{bedos2015fourier, bedos2016fourier, exel1997amenability, zeller1968produit}. 

A complete solution was established in \cite{crannNeufang2022Fejer} for locally compact groups, where it was shown that Fej\'{e}r representability in crossed products is equivalent to the approximation property, where a locally compact group $G$ has the approximation property (AP) if there exists a net $(v_i)$ in $A(G)$ such that $v_i \rightarrow 1$ in $\sigma(M_{cb}A(G), Q_{cb}(G))$ \cite{haagerup-Kraus1994approximation}. More specifically, it was shown in \cite[Theorem 4.1 and 4.10]{crannNeufang2022Fejer} that a locally compact group $G$ has the AP if and only if for any $W^*$--dynamical system $(M, G,\alpha)$, there exists a net $(h_j)$ in $A(G) \cap C_c(G)$ such that for all $T\in G\bar{\ltimes} M$,
$$ T=w^*-\lim_j\lim_i\int_G \frac{h_j(s)}{\Delta(s)}E(\lm(f_i)T\lm(f_i)\lm(s^{-1}))\lm(s) \ ds,$$
where $f_i\subseteq C_c(G)^+_{\|\cdot\|_1=1}$ is a symmetric bounded approximate identity for $\LO$, and $E$ is the canonical operator-valued weight. Similarly, a locally compact group $G$ has the AP if and only if for any $C^*$--dynamical system $(A, G,\alpha)$, there exists a net  $(h_j)$ in $A(G) \cap C_c(G)$ such that for all $T\in G\ltimes A$,
$$ T=\lim_j\lim_i\int_G \frac{h_j(s)}{\Delta(s)}E(\lm(f_i)T\lm(f_i)\lm(s^{-1}))\lm(s) \ ds.$$
In this paper, we will establish analogous results for $C^*$--algebraic and von Neumann algebraic crossed products over a discrete quantum group with the AP.

Anoussis, Katavolos, and Todorov studied the structure certain $VN(G)$--bimodules in $\BLT$ arising from closed left ideals in $L^1(G)$ \cite{Todorovanoussis2018bimodules}. Some of their results were extended from discrete groups to locally compact groups with the AP in \cite[Theorem 5.5]{crannNeufang2022Fejer}. As applications of our Fej\'{e}r representation, we generalize these results for $\LIQH$--bimodules in $\mc{B}(\ltq)$ of a discrete quantum group with the AP. Moreover, we characterize  $\LIQH$--bimodules in $\mc{B}(\ltq)$ which are invariant under a natural action of $M_{cb}(L^1(\widehat{\mathbb{G}}^{\rm op}))$.

As is well known, Fej\'{e}r representations in discrete crossed products are used in an important step towards Galois correspondences between bimodules in crossed products and subsets of groups (see, e.g., \cite{cs1,cs2,cs3,ku}). We postpone the application of our quantum group Fej\'{e}r representation to analogous Galois correspondences of discrete quantum group actions to subsequent work.

The structure of the paper is as follows. In Section 2, we recall some definitions and results needed in the theory, particularly those related to locally compact, discrete, and compact quantum groups, their actions on $C^*$-- or von Neumann algebras, and their crossed products.

Section 3 contains the main result of this paper, the Fej\'{e}r representation for elements of a crossed product.

In section 4, we define the $\LIQH$--bimodules $\mathrm{Bim}(J^{\perp})$ and $\mathrm{Ran}(J)$ for a closed left ideal $J\unlhd\LOQ$. We establish the equality $\mathrm{Bim}(J^{\perp})=\mathrm{Ran}(J)^{\perp}$ under the AP. Additionally, we show that under the AP, all weak*-closed invariant $\LIQH$--bimodules of $\mc{B}(\ltq)$ are in the form of  $\mathrm{Bim}(J^{\perp})$.  We show that jointly invariant weak*-closed subspaces of $\bltq$ are in the form of $\Sigma$--harmonic operators. Similarly, we establish a similar result for closed invariant $C(\wh{\G})$--bimodules of $\mc{K}(\ltq)$. We conclude this section with the notions of Fubini crossed product. And, as a final application of the Fej\'{e}r theorem, we show that any action of a discrete quantum group with the AP possesses a slice mapping property.

\section{Preliminaries}

In this section, we recall some basic definitions and results in the theory of locally compact quantum groups and their actions on $C^*$- and von Neumann algebras. The symbols $\mten$ and $\oten$ will denote the minimal and spatial tensor products of $C^*$-- and von Neumann algebras, respectively, while $\pten$ will denote the operator space projective tensor product. For any subset $X$ of a vector space $Y$, we denote the linear span of $X$ by $\la X\ra$. 
\subsection{Locally compact quantum groups}
For more details, we refer the readers to \cites{vaes2001thesis,kustermans2003locallyVN,woronowicz1998compact, van1996discrete}.
A locally compact quantum group $\G$ in the sense of Kustermans–Vaes is a
quadruple $\G=(\LIQ,\Gam, \vphi, \psi)$, where $\LIQ$ is a von Neumann algebra, $\Gam:\LIQ \rightarrow \LIQ \oten \LIQ $ is an injective normal unital $*$--homomorphism satisfying $(\Gam\ten\id)\Gam=(\id \ten \Gam)\Gam$ (co-associativity); and $\vphi$ and $\psi$ are respectively left and right Haar weights, meaning normal semifinite faithful weights on $\LIQ$ satisfying
$$\vphi((f\ten\id)\Gam(x))=f(1)\vphi(x),\, f \in \LOQ,\, x \in \LIQ_{\vphi}\:\: \text{(left\, invariance)},$$
and $$\psi((\id \ten f)\Gam(x))=f(1)\psi(x),\,  f \in \LOQ,\, x \in \LIQ_{\psi}\:\: \text{(right\, invariance)},$$
where above, and throughout the paper we denote the predual of $\LIQ$ by $\LOQ$. The co-multiplication $\Gam$ induces an associative completely contractive multiplication
\begin{equation}\label{eq1}
	\star:=\Gam_*: \LOQ \pten \LOQ\ni f_1\ten f_2 \mapsto f_1\star f_2:=\Gam_*(f_1\ten f_2)\in \LOQ.
\end{equation}
We denote the GNS Hilbert space of the left Haar weight $\vphi$ by $\LTQ$. There exists a left fundamental unitary $W\in \LIQ \oten \BLTQ $ satisfying the pentagonal relation $W_{12}W_{13}W_{23}=W_{23}W_{12}$. We have $\LIQ=\{(\id \ten \omega)(W) : \omega \in \TCQ\}''$, and for $x \in \LIQ$, the co-multiplication can be written as
 \begin{equation}\label{eq2}
\Gam(x)=W^*(1\ten x)W.
 \end{equation}
Similarly, there exists a right fundamental unitary operator V associated with the right Haar weight $\psi$ that satisfies the pentagonal relation and $\Gam(x)=V(x \ten1)V^*$.
The left regular representation is an injective completely contractive homomorphism $\lambda: \LOQ \rightarrow \BLTQ$ defined by $\lambda(f)=(f \ten \id)(W)$. We also obtain the right regular representation $\rho:\LOQ\rightarrow \BLTQ$ by $\rho(f)=(\id \otimes f)(V)$.

Defining $\LIQH:=\{\lambda(f) : f \in \LOQ\}''\subset \BLTQ$, we obtain the dual quantum group $\wh{\G}=(\LIQH, \wh{\Gam}, \wh{\vphi}, \wh{\psi})$ with co-multiplication $\wh{\Gam}$ given by $$\wh{\Gam}=\wh{W}^*(1\ten \wh{x})\wh{W},$$
where the dual fundamental unitary $\wh{W}=\sigma W^* \sigma$, and $\sigma$ is the flip map on $\LTQ\ten\LTQ$. In this case, we have $\LTQ\cong\LTQH$ and 
\begin{equation}\label{eq2.2.2}
    \overline{\la x\hat{x}\,:\,\, x \in \LIQ,\, \h{x}\in \LIQH \ra}^{w^*}=\BLTQ.
\end{equation}
Pontryagin duality ($\wh{\wh{\G}}=\G$) holds and we have $\LIQHH=\{(\wh{\lambda}(\hat{f}):=\hat{f}\ten \id)(\wh{W})\, :\, \hat{f} \in \LOQH\}=\LIQ$. Using the right regular representation $\rho$ we obtain a duality between $\LIQ$ and $\LIQHP$.
We will also denote the objects related to $\wh{\G}$ simply by using hats and objects related to $\wh{\G}'$ by using hats and primes.\\

An element $\h{b}$ in $\LIQH$ is called a completely bounded left multiplier of $\LOQ$ if $\h{b}\lambda(f)\in \lambda(\LOQ)$ for all $f \in \LOQ$ and the induced map
$$m_{\h{b}}^l: f \in \LOQ \rightarrow \lambda^{-1}(\h{b}\lambda(f)) \in \LOQ$$
is completely bounded on $\LOQ$ \cite{JNR-junge2009representation}. We let $M_{cb}^l(\LOQ)$  be the space of all completely bounded left multipliers of $\LOQ$. We can identify $\LOQ$ with the subalgebra $\lambda(\LOQ)$ in $M_{cb}^l(\LOQ)$.  Completely bounded right multipliers are defined analogously, where $M_{cb}^r(\LOQ)$ is the space of $\h{b}' \in \LIQH'$ such that $\rho(\LOQ)\h{b}' \subseteq\rho(\LOQ)$ and the induced map $m_{\h{b}'}^r$ is completely bounded. It turns out that $\McbQl$ and $\McbQr$ are dual Banach algebras and we denote the preduals by $\QcbQl$ and $\QcbQr$ \cite[\S 3]{huNeuRua2011cbmultiplierQ}. Moreover, we have the following representations for elements of $\QcbQl$ and $\QcbQr$ \cite[Prop. 3.2]{crann2017amenabilitycov2} and \cite[Prop.3.9]{daws2024approximation}.
\begin{gather}
\begin{aligned}\label{eq3}
    &\QcbQl=\{ \Omega_{x,\rho}\,|\; x\in C_0(\G)\mten\mc{K}(H),\, \rho \in \LOQ\pten\BH_* \};\\
    &\QcbQr=\{ \Omega_{x,\rho}\,|\;x\in C_0(\G)\mten\mc{K}(H),\, \rho \in \LOQ\pten\BH_* \};
\end{aligned} 
\end{gather}
where $\la  \wh{b} ,\Omega_{x,\rho}\ra=\la(\tl(\wh{b})\ten\id)x,\rho\ra$, for any $\wh{b}\in \McbQl$, and $\la \wh{b}', \Omega_{x,\rho}\ra=\la(\Theta^r(\wh{b}')\ten\id)x,\rho\ra$, for any $\wh{b}'\in \McbQr$.

The opposite and commutant von Neumann algebraic quantum groups $(\LIQ, \Gam)\op$  and $(\LIQ, \Gam)'$ are defined as follows \cite[\S 4]{kustermans2003locallyVN}. The underlying von Neumann algebra of the opposite quantum group is $\LIQ$ and the comultiplication is given by $\Gam\op(x)=\Sigma\Gam(x)$, for $x \in \LIQ$, where $\Sigma$ is the flip map on the tensor product. The underlying von Neumann algebra of the commutant quantum group is $\LIQ'$ and the comultiplication is $\Gam'(x)=(J\ten J)\Gam(JxJ)(J\ten J)$ for all $x' \in \LIQ'$, where $J:\LTQ\rightarrow \LTQ$ is the anti-unitary conjugation associated with the left Haar weight $\vphi$. For the opposite and commutant quantum groups $\G\op$ and $\G'$, we have the following relations;
$$\vphi\op =\psi \; \;\;\; \;\text{and} \;\;\;\; \; \vphi'(x')=\vphi(Jx'J), \;\; \, x'\in \LIQ'^+,$$
$$W\op=\Sigma V^* \Sigma  \; \;\;\; \;\text{and} \;\;\;\; \; W'=(J\ten J)W(J\ten J)=\wh{V},$$
With this notation we have the equality $(\wh{\G})\op=\wh{(\G')}$ \cite[Proposition 4.2]{kustermans2003locallyVN}. 

A reduced $C^*$--algebraic quantum group is a $C^*$--algebra $A$ with a non-degenerate $*$--homomorphism $\Gam: A \rightarrow M(A\mten A)$ such that 
\begin{itemize}
    \item $(\Gam\ten\id)\Gam=(\id \ten \Gam)\Gam$;
    \item $\la (\om \otimes \id)\Gam(a)\,|\,\om \in A^*,\, a\in A \ra$ and $\la (\id \otimes \om)\Gam(a)\,|\,\om \in A^*,\, a\in A \ra$ are norm dense in $A$;
    \item There exist faithful left/right invariant approximate KMS-weights on $(A, \Gam)$.
\end{itemize}

Given a locally compact quantum group $\G$, the reduced quantum group $C^*$--algebra of $\LIQ$ is defined as
$$C_0(\G):=\overline{\{(\id \ten \om )(W)\; :\; \om \in \TCQ\}}^{\lVert \cdot \rVert}.$$
The restriction of $\Gam$ to $C_0(\G)$ induces a reduced $C^*$--algebraic quantum group structure \cite[Proposition 1.6]{kustermans2003locallyVN}.

A unitary co-representation of a locally compact quantum group $\G$ is a unitary
operator $U\in M(C_0(\G) \ten \mathcal{K}(H_U))\subseteq \LIQ \oten \mathcal{B}(H_U) $ such that $$(\Gam\ten \id)(U)=U_{13}U_{23}.$$
Every unitary co-representation gives rise to a representation of $\LOQ$ via 
$$\pi_U: f\in \LOQ \mapsto (f \ten \id )(U)\in \mathcal{B}(H_U).$$
Two co-representations $U$ and $V$ are unitarily equivalent if there exists a unitary $u:H_V \rightarrow H_U$ such that 
$$(1\ten u^*)U(1\ten u)=V.$$
A co-representation $U$ irreducible if the corresponding representation $\pi_{U}:\LOQ \rightarrow \mathcal{B}(H_U)$ is irreducible. We let $\Irr$ denote the family of (equivalence classes of) irreducible unitary co-representations of $\G$.
\subsection{Compact and discrete quantum groups}
A locally compact quantum group is compact if $C_0(\G)$ is unital, or equivalently the left Haar weight is finite. When $\G$ is compact we denote $C_0(\G)$ by $C(\G)$.

Let $\G$ be compact with normalized Haar state $\vphi$, the following holds \cites{woronowicz1998compact, woronowicz1987compact-matrix-Pseudogroups}: 
\begin{itemize}
	\item Every irreducible co-representation $U^{\al}$ is finite-dimensional and is unitarily equivalent to a sub-representation of W,
	\item Every unitary co-representation of $\G$ can be decomposed into a direct sum of irreducible
	co-representations,
	\item If $u^{\al}\in \Irr$ and $\{e_i\,: i=1,...,n_{\al}\}$ is an orthonormal basis of $H_{\al}$, we obtain the matrix elements $u_{ij}^{\al}=(\id \ten \om_{ij})(u^{\al})\in \LIQ$ satisfying
	$$\Gam(u_{ij}^{\al})=\sum_{k=1}^{n_{\al}}u_{ik}^{\al}\ten u_{kj}^{\al}, \;\;\;\; 1\leq i, j \leq n_{\al},$$
	where $\om_{ij}=\om_{e_{j}, e_{i}}$ are vector functionals relative to the orthonormal basis $\{e_i\}$.
\end{itemize}
The linear space $\pol:=\la u_{ij}^\al\, :\, \al \in \Irr, 1\leq i, j \leq n_{\al}\ra$ is a unital Hopf
*-algebra which is norm dense in $C(\G)$ and weak* dense in $\LIQ$.

For every $\al \in \Irr$ there exists a positive invertible diagonal matrix $F^{\al }=\textnormal{diag}(\lambda_1^{\al}, ..., \lambda_{n_{\al}}^{\al})\in M_{n_{\al}}(\C)\cong \mathcal{B}(H_{\al})$ with $tr(F^{\al })=tr(F^{\al })^{-1}=:d_{\al}$ (the quantum dimension of $\al$) such that the Peter–Weyl orthogonality relations become (see \cite{daws2010OPBiprojectivity} for this convention):
\begin{equation}\label{eq4}
	\vphi((u^{\beta}_{kl})^*\uij)=\delta_{\al \beta}\delta_{ik}\delta_{jl}\frac{1}{\lambda_i^{\al}d_{\al}},\;\;\vphi(u^{\beta}_{kl}(\uij)^*)=\delta_{\al \beta}\delta_{ik}\delta_{jl}\frac{\lambda_j^{\al}}{d_{\al}}.
\end{equation}
For $x \in \LIQ$, we denote by $x\cdot\vphi$ and $\vphi\cdot x$ the elements in $\LOQ$ given by $\la x\cdot \vphi, y \ra=\la\vphi, yx \ra$ and  $\la \vphi \cdot x, y \ra=\la\vphi, xy \ra$. If $x=\uij$ for some $\al \in \Irr$ we use the notation $\vphi_{ij}^{\al}$ for $\uij\cdot\vphi$. Since $\pol$ is dense in $C(\G)$, it follows that $\pol \cdot \vphi$ is dense in $\LOQ$.\\
For $\al \in \Irr$, the character of $\al$ is
$$\chi^{\al}:=(\id\,\oten\, \tr)(u^{\al})=\sum_{i=1}^{n_{\al}}u_{ii}^{\al}\in \LIQ,$$
and the quantum character of $\al$ is
$$\chi^{\al}_q:=(\id\,\oten\, F^{\al })(u^{\al})=\sum_{i=1}^{n_{\al}}\lambda_i^{\al} u_{ii}^{\al}\in \LIQ.$$
For each quantum character $\chi^{\al}_q$, let $\varphi_q^{\al}:= \chi^{\al}_q\cdot \varphi$ be the $\LOQ$ element given by $\la \chi_q^{\al}\cdot \vphi, y \ra=\la\vphi, y\chi_q^{\al} \ra$. By the Peter–Weyl orthogonality relations we have 
$$\la\varphi_q^{\al} \star f,\,u_{kl}^{\beta *}\ra=\la f\star \varphi_q^{\al},\, u_{kl}^{\beta *} \ra=\la f, u_{kl}^{\beta *} \ra \frac{\delta_{\al \beta}}{d_{\al}}$$
for all $f \in \LOQ$ and $\beta \in \Irr$. Moreover, $\overline{\la \varphi_q^{\al}\,|\, \al \in \Irr \ra}$ is a closed ideal in the center, $\mathcal{Z}(\LOQ)$, of $\LOQ$.

We say that $\G$ is discrete if $\wh{\G}$ is compact, or equivalently if $\LOQ$ is unital. As is customary for a discrete quantum group $\G$, we write $\liq$ for $\LIQ$, $c_0(\G)$ for $C_0(\G)$,  $\loq$ for $\LOQ$, and $\ltq$ for $\LTQ$. We have
$$c_0(\G)=c_0-\!\!\!\bigoplus_{\al \in \Irrd}\mathcal{B}(H_{\al})=c_0-\!\!\!\bigoplus_{\al \in \Irrd}M_{n_{\al}}(\C),$$
$$\liq=\ell^{\infty}-\!\!\!\bigoplus_{\al \in \Irrd}\mathcal{B}(H_{\al})=\ell^{\infty}-\!\!\!\bigoplus_{\al \in \Irrd}M_{n_{\al}}(\C).$$
Moreover $\liq \cong M(c_0(\G))$, where $M(c_0(\G))$ is the multiplier algebra of the C*-algebra $c_0(\G)$.

For any locally compact quantum group $\G$, the fundamental unitary operators $W$ and $V$, induce complete contractions \cite{JNR-junge2009representation}
\begin{equation}\label{eq5}
	\tl:\LOQ \rightarrow \mathcal{CB}^{\sigma}(\BLTQ)\;:\; f \mapsto (f \ten \id )(W^*(1\ten \cdot)W),
\end{equation}
and
\begin{equation*}
	\Theta^r:\LOQ\rightarrow \mathcal{CB}^{\sigma}(\BLTQ) \;:\; f \mapsto (\id \ten f)(V(\cdot \ten 1)V^*).
\end{equation*}
Moreover, $\tl$ is an anti-homomorphism from $\LOQ$ into $\mc{CB}_{L^{\infty}(\hat{\mathbb{G}}')}^{\sigma, L^{\infty}(\mathbb{G})}(\mc{B}(L_2(\mathbb{G})))$ and $\Theta^r$ is a homomorphism from $\LOQ$ into $\cb{\LIQH}{\LIQ}{\BLTQ}$, but they are not necessarily onto. In the discrete case, they are both bijective \cite[Theorem 3.4, 3.5]{JNR-junge2009representation}:
\begin{thm}\label{thm1}
	Let $\G$ be a discrete quantum group. Then the map
	\begin{itemize}[leftmargin=2em]
	\item  $\tl$ is a completely isometric algebra anti-isomorphism from $\loq$ onto $\mc{CB}_{L^{\infty}(\hat{\mathbb{G}}')}^{\sigma,\ell^{\infty}(\mathbb{G})}(\bltq)$,
		\item $\Theta^r$  is a completely isometric algebra isomorphism from $\loq$ onto $\cb{\LIQH}{\LIQ}{\bltq}$.
	\end{itemize}
	
\end{thm}
\begin{remark}\cite[Theorem 4.10]{JNR-junge2009representation}
	In general, for a locally compact quantum group $\G$, the maps $\tl$ and $\Theta^r$ extend  respectively to a completely isometric anti-isomorphism and a completely isometric isomorphism of completely contractive Banach algebras:
	$$\tl:M_{cb}^l(\LOQ)\cong \cb{L^{\infty}(\hat{\mathbb{G}}')}{\LIQ}{\BLTQ}, $$
	$$\Theta^r:M_{cb}^r(\LOQ)\cong \cb{\LIQH}{\LIQ}{\BLTQ}.$$
\end{remark}
By the equality $(\wh{\G})\op=\wh{(\G')}$, for $\wh{\G}\op$ we have $\wh{\tl}\op:M_{cb}^l(\LOQHOP)\cong \cb{\LIQ}{\LIQHOP}{\BLTQ}$, and for any $x \in \BLTQ$
\begin{equation}\label{eq6}
	\wh{\tl}\op(\wh{f})(x)= (\wh{f}\ten \id)(\ww^*(1\ten x)\ww),
\end{equation}
where above and throughout we use $\widetilde{W}$ for $\widehat{W}^{\op}$ to simplify the notation.
\begin{remark}\label{r:comm}
    One can show that for any $f\in M_{cb}^r(\LOQ)$ and $\wh{f}\in M_{cb}^l(\LOQHOP)$, the operators $\wh{\tl}\op(\wh{f})$ and $\Theta^r(f)$ commute (see \cite[Theorem 5.1]{JNR-junge2009representation}). Specifically, let $x\in \LIQ$ and $\wh{x}\in \LIQHOP$, then 
    \begin{align*}
        \trf(\tll{f}(x\wh{x}))&=\trf(x\underbrace{\tll{f}(\wh{x})}_{\in \LIQHOP})=\underbrace{\trf(x)}_{\in \LIQ}\tll{f}(\wh{x})\\
       &=\tll{f}(\trf(x)\wh{x})=\tll{f}(\trf(x\wh{x})). 
    \end{align*}
      By normality of $\trf$ and $\tll{f}$ and density (Equation \eqref{eq2.2.2}), we conclude that $$\trf\tll{f}=\tll{f}\trf.$$
\end{remark}
\subsection{Actions and crossed products}
\begin{defn}
	Let $\n$ be a von Neumann algebra. A left action of $\G$ on $\n$ is  a normal, injective unital *-homomorphism $\al :\n \rightarrow \LIQ \oten \n$ satisfying
	$$(\Gam \ten \id)\al=(\id \ten \al)\al.$$
\end{defn}

We define the fixed point algebra $\n^{\al}$ to be $\n^{\al}:=\{x\in \n \;|\; \al(x)=1\ten x\}$, then $\n^{\al} $ is a von Neumann subalgebra of $\n.$

\begin{prop}\cite[Proposition 2.3.3]{vaes2001thesis}
	Let $\al$ be a left action of $\G$ on $\n$. For every $x \in \n^+$,  the element $E_{\al}(x)=(\psi \ten \id )\al(x)$ of $\n_{\ext}^+$ belongs to $(\n^{\al})^+_{\ext}$. Moreover, $E_{\al}:\n^+\rightarrow (\n^{\al})^+_{\ext}\subset \n_{\ext}^+$  is a normal, faithful operator valued weight.
\end{prop}
An action $\al$ of $\G$ on a von Neumann algebra $\n$ induces a right $\LOQ$--module structure on $\n$ and a left $\LOQ$--module structure on $\n_*$ defined  by
\begin{equation}
    \begin{aligned}
        & n\star_{\al}f:=(f\ten \id )\al(n),\; n\in \n\, f\in \LOQ,\\
        &f \leftindex_\al {\star}\, \omega:=(f\ten \omega)\circ \al,\; \omega\in \n_*,\, f\in \LOQ.
    \end{aligned}
\end{equation}

\begin{remark}\label{rem2.5.3}
\begin{enumerate}
    \item For any $n\in \n $, $\omega\in \n_*$,  and $f\in \LOQ$ we have
\[\la n\star_{\al}f,\, \om\ra=\la n, \,f \leftindex_\al {\star}\,\om\ra.\]
\item Let $\LOQ$ be unital ($\G$ is discrete) and denote the unit by $\epsilon$. Then $\epsilon\leftindex_{\al}{\star} \,\omega=\omega$ for any $\omega \in \n_*$. This is easily seen on elements of the form $\omega=f \leftindex_{\al}{\star}\om'$, which span a dense set in view of surjectivity of the module map $\alpha_*:\ell^1(\G)\pten\n_*\to\n_*$. Consequently, for any $n \in \n$, $n\star_{\al}\epsilon=n$.
\end{enumerate}  
\end{remark}
\begin{defn}
	Let $\al$ be an action of $\G$ on $\n$. The crossed product of $\n$ with respect to $\al$ is a von Neumann subalgebra of $\BLTQ \oten \n$ generated by $\al(\n)$ and $\LIQH$. We denote this crossed product by $\crs{\G}{\n}$. So, we have $$\crs{\G}{\n}=(\al(\n)\cup \LIQH\ten \C)''.$$
\end{defn}
There is a unique action $\wh{\al}$ of $\wh{\G}\op$ on the crossed product $\crs{\LIQ}{\n}$ such that 
\begin{gather}
	\begin{aligned}\label{eq7}
		\wh{\al}(\al(x))&=1 \ten \al(x)\;\; \;\;\text{for all}\; x\in \n,  \\
		\wh{\al}(a\ten 1)&=\wh{\Gam}\op(a)\ten 1 \;\; \text{for all}\; a\in \wh{\m}\op=\LIQHOP.
	\end{aligned}
\end{gather}
Moreover the fundamental unitary $\widetilde{W}$ of $\wh{\G}\op$ implements the action via
\begin{equation}\label{eq8}
	\wh{\al}(z)=(\ww^* \ten 1)(1 \ten z)(\ww \ten 1), \ \ z\in \LIQ\, {}_{\al}\ltimes \n.
\end{equation}
It turns out that the operator-valued weight corresponding to $\wh{\al}$ is semi-finite \cite[Theorem 2.4.6]{vaes2001thesis}:
\begin{thm}\label{thm6}
	We have 
	\begin{align*}
		(\crs{\G}{\n})^{\wh{\al}}=\al(\n)=\{z \in \LIQ \oten \n \;|\;(\id\ten \al)(z)=(\Gam\ten \id)(z)\}.
	\end{align*}
Furthermore, $E_{\wh{\al}}:(\crs{\G}{\n})^+\ni x \mapsto (\wh{\psi}\op\ten \id)\wh{\al}(x)\in((\crs{\G}{\n})^{\wh{\al}})_{\ext}^+$ is a normal, faithful, semi-finite operator valued weight.
\end{thm}
Since the right Haar weight $\wh{\psi}\op$ of $\wh{\G}\op$ is equal to the left Haar weight $\wh{\vphi}$ of $\wh{\G}$, we write $E_{\wh{\al}}(x)=(\wh{\vphi}\ten \id)\wh{\al}(x)$. When $\al$ is clear from context, we often simply write $E$ instead of $E_{\wh{\al}}$.
\begin{remark}
    If $\G$ is discrete then $E_{\wh{\al}}$ is a conditional expectation from the crossed product
     $\G \prescript{}{\al}\ltimes\, \n$ onto the fixed point algebra of $\wh{\al}$ i.e. $\al(\n)$.
\end{remark}
The dual action $\wh{\al}$ gives a right $\LOQHOP$--module structure on the crossed product $\crs{\G}{\n}$ via
$$T\cdot \wh{f}:=(\wh{f}\ten \id \ten \id)(\wh{\al}(T)).$$
It follows from equations \eqref{eq6} and \eqref{eq8} that
\begin{equation}\label{eqbeforeDef2.10}
T\cdot \wh{f}=(\wh{\tl}\op(\wh{f}) \ten \id)(T), \;\; \wh{f} \in \LOQHOP,\;\;T \in \crs{\G}{\n}.
\end{equation}

\begin{defn}
    For a discrete quantum group $\G$, a left action of $\G$ on a $C^*$--algebra $A$ is a non-degenerate $*$--homomorphism $\al: A \rightarrow M(c_0(\G)\mten A)$ satisfying
    \begin{itemize}
        \item $(\Gam \ten \id)\al=(\id \ten \al)\al$;
        \item $\la (c_0(\G)\otimes 1)\al(A)\ra$ is norm dense in $c_0(\G)\mten A$ (the Podle\'{s} condition). 
    \end{itemize}
\end{defn}
Then $\overline{\la \al(A)(C(\wh{\G})\otimes 1) \ra}^{\lVert \cdot \rVert}$ is a $C^*$--algebra \cite{vaes2005newApproach}. For an action $\al$ of a discrete quantum group $\G$ on a $C^*$--algebra $A$, the reduced crossed product is the $C^*$--algebra generated by $\al(A)$ and $C(\wh{\G})$. We denote the reduced crossed product by
$$\G \leftindex_\al {\ltimes}_r A:=\overline{\la \al(A)(C(\wh{\G})\otimes 1) \ra}^{\lVert \cdot \rVert}\subseteq M(\kltq \mten A).$$
Using the universal representation of $A$, we consider $\al(A)\subseteq M(c_0(\G)\mten A)$ contained in $\liq \oten A^{**}$ and $\G \leftindex_\al {\ltimes}_r A$ contained in $\bltq\oten A^{**}$.  

\section{Fej\'{e}r representations}
In this section we obtain a Fej\'{e}r representation for any elements in $C^*$-- and von Neumann algebraic crossed products by a discrete quantum group with the approximation property. 

A locally compact quantum group has the approximation property (AP) if there
exists a net $(\wh{f_{\iota}})\subseteq\LOQH$ such that $\wh{\tl}(\wh{f_{\iota}})\rightarrow \id_{\LIQH}$ in the stable point-weak* topology, i.e., for any (separable) Hilbert space $H$, $\wh{\tl}(\wh{f_{\iota}})\ten\id_{\BH} \rightarrow \id_{\LIQH\oten\mc{B}(H)}$ in the weak* topology (see, e.g., \cite{crann2019inneramenability,krausRuan1999AP-Kac,daws2024approximation}). As shown by Haagerup-Kraus \cite[Propposition 1.7]{haagerup-Kraus1994approximation} we may replace $\BH$ with a von Neumann algebra $\n$ and obtain the weak* convergence of $\wh{\tl}(\wh{f_{\iota}})\ten\id_{\n} \rightarrow \id_{ \LIQH\oten\n}$. Daws-Krajczok-Voigt in \cite[Proposition 4.3]{daws2024approximation} proved that $\G$ has the AP if and only if $\G'$ has the AP. Therefore, we can assume that $\wh{\tl}\op(\wh{f_{\iota}})\rightarrow \id_{\LIQHOP}$ in stable point-weak* topology. They also showed that the AP is equivalent to the existence of a net $(\wh{f_{\iota}})\subseteq\LOQH$ such that $\wh{f_{\iota}}\rightarrow 1$ in the weak* topology of $\McbQHl$ \cite[Theorem 4.4]{daws2024approximation}.

We require the following lemma.

\begin{lem}\label{lem3.3}
     Let $\G$ be a discrete quantum group. Then for any  $\wh{f}\in \polh\cdot \wh{\vphi}:=\la\wh{\vphi}_{ij}^{\beta}\, :\, \beta \in \Irrd, i,\,  j =1,\cdots n_{\beta}\ra$ there is a finite subset $\mc{F}\subseteq \Irrd$ such that for any action $\alpha$ of $\G$ on a von Neumann algebra $\n$
    \begin{equation}\label{eq9lem3.3}
       (\tf\ten \id)(T)= \sum_{\beta\in \mc{F}}\sum _{i,j,k=1}^{n_{\beta}}\frac{d_{\beta}}{\lambda^{\beta}_i}\la \wh{f}, \wh{u}_{ji}^{\beta} \ra E(T((\wh{u}_{ki}^{\beta})^*\otimes 1))(\wh{u}_{kj}^{\beta}\otimes 1)
    \end{equation}
    for all $T\in \crs{\G}{\n}$.
\end{lem}
\begin{proof}
	Let  $\beta \in \Irr$ and $i, j, k \in \{1, \cdots n_{\beta}\}$. Then
 \begin{align}
		E(T(\uis{ki}{\beta}\ten1))(\uih{kj}{\beta}\ten 1)&=[(\wh{\vphi}\ten \id \ten \id)\wh{\al}(T(\uis{ki}{\beta}\ten 1))](\uih{kj}{\beta}\ten 1) \nonumber\\
		&=[(\wh{\vphi}\ten \id \ten \id)\wh{\al}(T)\wh{\al}(\uis{ki}{\beta}\ten 1)](\uih{kj}{\beta}\ten 1) \nonumber\\
	\eqref{eq7}	\;\;\;\;\;\;\;\;\;\;\;\;\;\;\;\;\;\;\;\;&=[(\wh{\vphi}\ten \id \ten \id)\wh{\al}(T)(\wh{\Gam}\op(\uis{ki}{\beta}))\ten 1)](\uih{kj}{\beta}\ten 1) \nonumber\\
		&=[(\wh{\vphi}\ten \id \ten \id)\wh{\al}(T)(\sum_{l=1}^{n_\beta}\Sigma(\uis{kl}{\beta}\ten\uis{li}{\beta})\ten 1)](\uih{kj}{\beta}\ten 1) \nonumber\\
			&=[(\wh{\vphi}\ten \id \ten \id)\wh{\al}(T)(\sum_{l=1}^{n_\beta}(\uis{li}{\beta}\ten\uis{kl}{\beta})\ten 1)](\uih{kj}{\beta}\ten 1) \nonumber\\
				&=\sum_{l=1}^{n_\beta}[(\wh{\vphi}\ten \id \ten \id)\wh{\al}(T)(\uis{li}{\beta}\ten 1\ten 1)](\uis{kl}{\beta}\uih{kj}{\beta}\ten 1) \nonumber\\
			&=\sum_{l=1}^{n_\beta}[(\uis{li}{\beta}\cdot\wh{\vphi}\ten \id \ten \id)\wh{\al}(T)](\uis{kl}{\beta}\uih{kj}{\beta}\ten 1) \nonumber\\
		\eqref{eqbeforeDef2.10}	\;\;\;\;\;\;\;\;\;\;\;\;\;\;\;\;\;\;\;\;
		&=\sum_{l=1}^{n_\beta}[(\wh{\tl}\op(\uis{li}{\beta}\cdot\wh{\vphi})\ten \id )(T)](\uis{kl}{\beta}\uih{kj}{\beta}\ten 1)\nonumber
	\end{align}
	Therefore 
 \begin{equation}\label{eq10}
     E(T((\wh{u}_{ki}^{\beta})^*\otimes 1))(\wh{u}_{kj}^{\beta}\otimes 1)=\sum_{l=1}^{n_\beta}[(\wh{\tl}\op(\uis{li}{\beta}\cdot\wh{\vphi})\ten \id )(T)](\uis{kl}{\beta}\uih{kj}{\beta}\ten 1).
 \end{equation}

By assumption, $\wh{f}\in \la\wh{\vphi}_{ij}^{\beta}\, :\, \beta \in \Irrd, i,\,  j =1,\cdots n_{\beta}\ra$ admits a decomposition of the form
$$\wh{f}=\sum_{\beta \in \mc{F}}\sum_{i,j=1}^{n_{\beta}}c_{ij}^{\beta}(\uis{ij}{\beta}\cdot \wh{\varphi}),$$
where $\mc{F}$ is a finite subset of $\Irrd$ and $c_{ij}^\beta\in\mathbb{C}$. By the orthogonality relations \eqref{eq4}, for $\gamma \in\mathcal{F}$,
\begin{align*}
	\la \wh{f}, \uih{kl}{\gamma} \ra&=\la \sum_{\beta \in \mc{F}}\sum_{i,j=1}^{n_{\beta}}c_{ij}^{\beta}(\uis{ij}{\beta}\cdot \wh{\varphi}), \uih{kl}{\gamma}\ra\\
    &=\sum_{\beta \in \mc{F}}\sum_{i,j=1}^{n_{\beta}}c_{ij}^{\beta}\wh{\vphi}(\uih{kl}{\gamma}\uis{ij}{\beta})\\
	&=\sum_{\beta \in \mc{F}}\sum_{i,j=1}^{n_{\beta}}c_{ij}^{\beta} \delta_{\beta\gamma}\delta_{ik}\delta_{jl}\frac{\lambda_j^{\beta}}{d_{\beta}}\\
    &=c_{kl}^{\gamma}\frac{\lambda_l^{\gamma}}{d_{\gamma}}.
\end{align*}
So, \begin{equation}\label{eq11}
	\wh{f}=\sum_{\beta \in \mc{F}}\sum_{i,j=1}^{n_{\beta}}\frac{d_{\beta}}{\lambda_j^{\beta}}\la \wh{f}, \uih{ij}{\beta} \ra(\uis{ij}{\beta}\cdot \wh{\varphi}).
\end{equation}
Now, using \eqref{eq10} and \eqref{eq11} we have
	\begin{align}\label{eq12theta}
		\sum_{\beta\in \mc{F}}\sum _{i,j,k=1}^{n_\beta}\frac{d_{\beta}}{\lambda^{\beta}_i}\la \wh{f}, \wh{u}_{ji}^{\beta} \ra &E(T((\wh{u}_{ki}^{\beta})^*\otimes 1))(\wh{u}_{kj}^{\beta}\otimes 1)\nonumber\\
		&=\sum_{\beta\in \mc{F}}\sum _{i,j,k,l=1}^{n_\beta}\frac{d_{\beta}}{\lambda^{\beta}_i}\la \wh{f}, \wh{u}_{ji}^{\beta} \ra (\wh{\tl}\op (\uis{li}{\beta}\cdot \phih)\ten  \id)(T)(\uis{kl}{\beta}\uih{kj}{\beta}\ten 1)\nonumber\\
		&=\sum_{\beta\in \mc{F}}\sum _{i,j,l=1}^{n_\beta}\frac{d_{\beta}}{\lambda^{\beta}_i}\la \wh{f}, \wh{u}_{ji}^{\beta} \ra (\wh{\tl}\op (\uis{li}{\beta}\cdot \phih)\ten  \id)(T)(\sum_{k=1}^{n_\beta}\uis{kl}{\beta}\uih{kj}{\beta}\ten 1)\nonumber\\
		&=\sum_{\beta\in \mc{F}}\sum _{i,j,l=1}^{n_\beta}\frac{d_{\beta}}{\lambda^{\beta}_i}\la \wh{f}, \wh{u}_{ji}^{\beta} \ra (\wh{\tl}\op (\uis{li}{\beta}\cdot \phih)\ten  \id)(T)(\delta_{lj}\ten 1)\nonumber\\
		&=\sum_{\beta\in \mc{F}}\sum _{i,j=1}^{n_\beta}\frac{d_{\beta}}{\lambda^{\beta}_i}\la \wh{f}, \wh{u}_{ji}^{\beta} \ra (\wh{\tl}\op (\uis{ji}{\beta}\cdot \phih)\ten  \id)(T)\nonumber\\
		&=\wh{\tl}\op (\sum_{\beta\in \mc{F}}\sum _{i,j=1}^{n_\beta}\frac{d_{\beta}}{\lambda^{\beta}_i}\la \wh{f}, \wh{u}_{ji}^{\beta} \ra\uis{ji}{\beta}\cdot \phih)\ten  \id)(T)\nonumber\\
		&=(\tf\ten \id)(T).
	\end{align}
 This proves the lemma.
\end{proof}
\begin{thm} \label{thm3.5}
	Let $\G$ be a discrete quantum group. Then the following are equivalent
 \begin{enumerate}
     \item $\G$ has the AP.
     \item There exist nets $(\wh{f_{\iota}})_{\iota \in I}\subseteq \polh\cdot \wh{\vphi} \subseteq \LOQH$ and $(\mc{F}_{\iota})_{\iota\in I}$,consisting of finite subsets of $\Irrd$, such that for any action $\alpha$ of $\G$ on a von Neumann algebra $\n$ and any $T\in \crs{\G}{\n}$, we have
	\begin{equation}\label{eq13fejer}
	    T=w^*-\lim_{\iota}\sum_{\beta\in \mc{F}_{\iota}}\sum _{i,j,k=1}^{n_\beta}\frac{d_{\beta}}{\lambda^{\beta}_i}\la \wh{f_{\iota}}, \wh{u}_{ji}^{\beta} \ra E(T((\wh{u}_{ki}^{\beta})^*\otimes 1))(\wh{u}_{kj}^{\beta}\otimes 1),
	\end{equation}
 \end{enumerate} 
 where $E:=E_{\h{\al}}$ is the operator valued weight induced by $\h{\al}$.
\end{thm}
\begin{proof}
Suppose $\G$ has the AP, then there is a net $(\wh{f_{\iota}})_{\iota \in I}\subset \LOQH $ such that $\wh{\tl}\op(\wh{f_{\iota}})|_{\LIQHOP} \rightarrow \id_{\LIQHOP}$ in the stable point-weak* topology.
Since $\polh\cdot \wh{\varphi}$ is dense in $\LOQH$, for each $\iota\in I$ and given $m\in \N$, there is an element $\wh{f}_{\iota, m}\in \polh\cdot \wh{\varphi}$ such that $\|\wh{f}_{\iota}-\wh{f}_{\iota, m}\|_{\LOQ}\leq\frac{1}{m}$. If the net $I\times \N$ is given the product ordering, then  $(f_{\iota, m}:\, \iota\in I, m\in \N)$ is a net in $\polh\cdot \wh{\varphi}$ which converges to 1 in the weak* topology of $\McbQHl$. Therefore, without loss of generality, we assume that $(\wh{f_{\iota}})\subset \polh\cdot \wh{\varphi}$. By Lemma \ref{lem3.3}, for each $\wh{f_{\iota}}$, there exists a finite set $\mc{F}_{\iota}\subseteq\Irrd$ such that for any action $\alpha$ of $\G$ on a von Neumann algebra $\n$, and any $T\in \crs{\G}{\n}$, we have
    \begin{equation*}
       (\tlf\ten \id)(T)= \sum_{\beta\in \mc{F}_{\iota}}\sum _{i,j,k=1}^{n_\beta}\frac{d_{\beta}}{\lambda^{\beta}_i}\la \wh{f_{\iota}}, \wh{u}_{ji}^{\beta} \ra E(T((\wh{u}_{ki}^{\beta})^*\otimes 1))(\wh{u}_{kj}^{\beta}\otimes 1).
    \end{equation*}
	It remains to show that $w^*-\lim_{\iota} (\tlf\ten \id)(T)=T$. Write $\Phi(\wh{f_{\iota}}):=\wh{\tl}\op(\wh{f_{\iota}})|_{\LIQHOP}$. Then, by the AP,
    $$w^*-\lim_{\iota}(\Phi(\wh{f_{\iota}})\ten \id_{\BLTQ\oten \n})=\id_{\LIQHOP\oten (\BLTQ\oten \n)}.$$
    Let $T \in \BLTQ\oten \n$ and fix a state $\om\in \TCQ$. By the (left version of the) proof of \cite[Proposition 4.3]{JNR-junge2009representation}, we have 
	$$(\tlf\ten\id_{\n})(T)=(\om\ten\id_{\BLTQ\oten \n}) (\widetilde{W}_{12}((\Phi(\wh{f_{\iota}})\ten \id_{\BLTQ\oten \n})\widetilde{W}^*_{12}T_{23}\widetilde{W}_{12})\widetilde{W}^*_{12}).$$
	Since $\widetilde{W}^*_{12}T_{23}\widetilde{W}_{12}\in \LIQHOP\oten(\BLTQ\oten\n)$, by stable point-weak* convergence, for any $\rho\in (\BLTQ\oten\n)_*$,
	\begin{align*}
		\lim_{\iota} \la (\tlf\ten\id)(T), \rho\ra &=\lim_{\iota} \la (\om\ten\id_{\BLTQ \oten \n}) (\widetilde{W}_{12}((\Phi(\wh{f_{\iota}})\ten \id_{\BLTQ\oten \n})\widetilde{W}_{12}^*T_{23}\widetilde{W}_{12})\widetilde{W}_{12}^*), \rho \ra \\
		&=\lim_{\iota} \la \widetilde{W}_{12}((\Phi(\wh{f_{\iota}})\ten \id_{\BLTQ\oten \n})\widetilde{W}_{12}^*T_{23}\widetilde{W}_{12})\widetilde{W}_{12}^* , \om\ten \rho \ra\\
		&=\la \widetilde{W}_{12}((\id_{\LIQH\oten \BLTQ\oten \n})\widetilde{W}_{12}^*T_{23}\widetilde{W}_{12})\widetilde{W}_{12}^* ,\om \ten \rho\ra\\
		&=\la  \widetilde{W}_{12}(\widetilde{W}_{12}^*T_{23}\widetilde{W}_{12})\widetilde{W}_{12}^*,  \om \ten \rho\ra\\
		&=\la T,  \rho \ra.
	\end{align*}
Hence, for any $T\in \BLTQ \oten \n$, $w^*-\lim_{\iota} (\tlf \ten \id_{\n})(T)=T.$

Conversely, suppose that there exist nets $(\wh{f_{\iota}})\subseteq \polh\cdot\wh{\vphi} \subseteq L^1(\hat{\G})$ and $(\mc{F}_{\iota})$, such that for any action $\alpha$ of $\G$ on a von Neumann algebra $\n$, equation \eqref{eq13fejer} holds. Let $\G$ act trivially on $\BH$ for a Hilbert space $H$, then $\crs{\G}{\BH}=\LIQH\oten \BH$. From the proof of Lemma \ref{lem3.3}, precisely \eqref{eq12theta}, for any $T\in \LIQH\oten \BH$ and each $\iota \in I$ it follows
$$\sum_{\beta\in \mc{F}_{\iota}}\sum _{i,j,k=1}^{n_\beta}\frac{d_{\beta}}{\lambda^{\beta}_i}\la \wh{f_{\iota}}, \wh{u}_{ji}^{\beta} \ra E(T((\wh{u}_{ki}^{\beta})^*\otimes 1))(\wh{u}_{kj}^{\beta}\otimes 1)=(\tlf\ten \id)(T).$$
Thus $T=w^*-\lim_{\iota} (\tlf\ten \id)(T)$, implying that $\tlf$ converges to $\id_{\LIQHOP}$ in the stable point-weak*-topology. Therefore, $\G$ has the AP.
\end{proof}

We now establish the analogous result in the $C^*$-setting. 

Let $\al: A \rightarrow M(c_0(\G)\mten A)\subseteq \liq\oten A^{**}$ be a left action of a discrete quantum group $\G$ on a $C^*$--algebra A. Then $\al^*|_{(\liq \oten A^{**})_*}: \loq\pten A^*\rightarrow A^*$, and $\widetilde{\al}:=(\al^*|_{(\liq \oten A^{**})_*})^*:A^{**} \rightarrow \liq \oten A^{**}$ defines a left action of $\G$ on the von Neumann algebra $A^{**}$. $\G \leftindex_\al {\ltimes}_r A$ is weak* dense in $\G \leftindex_{\widetilde{\al}} {\ltimes} A^{**}$, since $\widetilde{\al}|_{A}=\al$ and  $\widetilde{\al}$ is normal. 


Let $T=\al(a)(\wh{u}_{ij}^{\beta}\otimes 1) \in\G \leftindex_\al {\ltimes}_r A$ for some $a \in A$, $\beta \in \Irrd$ and $i,j \in \{1,\cdots ,n_{\be}\}$. Then for any $\gamma \in \Irrd$
\begin{align}\label{e:align}
    E_{\wh{\wtal}}(T(\uis{kl}{\gamma}\otimes 1))&=(\wh{\vphi}\ten \id \ten \id)\wh{\wtal}(T(\uis{kl}{\gamma}\ten 1))\\
    &=(\wh{\vphi}\ten \id \ten \id)[(\ww^*\otimes1)(1\otimes \al(a)(\uih{ij}{\be}\otimes 1)(\uis{kl}{\gamma}\otimes1))(\ww\otimes 1)]\\
    &=\al(a) (\wh{\vphi}\ten \id \ten \id)[(\ww^*\otimes1)(1\otimes (\uih{ij}{\be}\otimes 1)(\uis{kl}{\gamma}\otimes1))(\ww\otimes 1)]\\
    &=\wh{\vphi}(\uih{ij}{\be}\uis{kl}{\gamma}) \al(a) \in \al(A).
\end{align}
Hence, by linearity and density, $ E_{\wh{\wtal}}(T(\uis{kl}{\gamma}\otimes 1))\in \al(A)$ for any $T\in\G \leftindex_\al {\ltimes}_r A$.

Similar to \eqref{eqbeforeDef2.10}, the dual action $\wh{\wtal}$ gives a right $\LOQHOP$--module structure on the reduced crossed product $\crsred{\G}{A}$ by 
$$T\cdot \wh{f}:=(\wh{f}\ten \id \ten \id)(\wh{\wtal}(T))=(\wh{\tl}\op(\wh{f}) \ten \id)(T), \;\; \wh{f} \in \LOQHOP,\;\;T \in \crsred{\G}{A}.$$


\begin{lem}\label{lem3.4}
    Let $\G$ be a discrete quantum group and $\al$ be an action of $\G$ on a  $C^*$--algebra $A$. For any $\gamma \in \Irrd$ and  $\phi \in (\crsred{\G}{A})^*$, $(\wh{\tl}\op(\chi_{\gamma})\otimes\id)^*(\phi)\in (\G \leftindex_{\widetilde{\al}} {\ltimes} A^{**})_*$, where $\chi_{\gamma}:=d_{\gamma} \wh{\varphi}_q^{\gamma}$.
\end{lem}

\begin{proof}
First consider $\al(a)(\uis{ij}{\be}\otimes 1)$ for some $a\in A$ and $\be \in \Irrd$.
Since $\wh{\tl}\op(\chi_\gamma)$ is an $\liq$--bimodule map and $\al(a)\in\liq\oten A^{**}$ lies in $\liq$,
        $$(\wh{\tl}\op(\chi_{\gamma})\otimes \id)(\al(a)(\uis{ij}{\beta}\otimes 1))=\al(a)(\wh{\tl}\op(\chi_{\gamma})(\uis{ij}{\beta})\otimes 1).$$ But
        \begin{align*}
            \wh{\tl}\op(\chi_{\gamma})(\uis{ij}{\beta})&=(\chi_{\gamma}\otimes \id)(\wh{\Gam}\op(\uis{ij}{\be}))\\
            &=(\chi_{\gamma}\otimes \id)(\sum_{k=1}^{n_{\beta}}\uis{kj}{\be}\otimes \uis{ik}{\be})\\
            &=d_\gamma(\wh{\varphi}\otimes \id)(\sum_{k=1}^{n_{\beta}}\sum_{l=1}^{n_{\gamma}}\lambda^\gamma_l\uis{kj}{\be}\uih{ll}{\gamma}\otimes \uis{ik}{\be}))\\
            &=d_{\gamma}(\sum_{k=1}^{\beta}\sum_{l=1}^{\gamma}\lambda^\gamma_l\delta_{\be \gamma}\delta_{kl}\delta_{jl}\frac{1}{\lambda^\gamma_ld_{\gamma}})\uis{ik}{\be})).
        \end{align*}
        Therefore,
        \begin{equation}\label{eq14}
            (\wh{\tl}\op(\chi_{\gamma})\otimes \id)(\al(a)(\uis{ij}{\beta}\otimes 1))=\begin{cases} 
      \al(a)(\uis{ij}{\be}\otimes 1) & \text{if}\,\; \be=\gamma, \\
      0 &  \text{if}\,\; \be\neq\gamma. \\ 
   \end{cases}
        \end{equation}
   Next, we observe that
   \begin{align*}
       \sum_{k,l=1}^{n_{\gamma}}\lambda^\gamma_kd_{\gamma}E_{\wh{\wtal}}&\left( (\wh{\tl}\op(\chi_{\gamma})\otimes \id)(\al(a)(\uis{ij}{\gamma}\otimes 1))(\uih{kl}{\gamma}\otimes 1)\right)(\uis{kl}{\gamma}\otimes 1)\\
       &= \sum_{k,l=1}^{n_{\gamma}}\lambda^\gamma_kd_{\gamma}E_{\wh{\wtal}}\left( (\al(a)(\uis{ij}{\gamma}\otimes 1))(\uih{kl}{\gamma}\otimes 1)\right)(\uis{kl}{\gamma}\otimes 1)\\
       &=\sum_{k,l=1}^{n_{\gamma}}\lambda^\gamma_kd_{\gamma}\left[(\wh{\varphi}\otimes\id \otimes\id)\wh{\wtal}\left( (\al(a)(\uis{ij}{\gamma}\otimes 1))(\uih{kl}{\gamma}\otimes 1)\right)\right](\uis{kl}{\gamma}\otimes 1)\\
       &=\sum_{k,l=1}^{n_{\gamma}}\lambda^\gamma_kd_{\gamma}\al(a)\left((\wh{\varphi}\otimes\id \otimes\id)(\wh{\Gamma}\op(\uis{ij}{\gamma}\uih{kl}{\gamma})\otimes 1)\right)(\uis{kl}{\gamma}\otimes 1)\\
       &=\sum_{k,l=1}^{n_{\gamma}}\lambda^\gamma_kd_{\gamma}\al(a)\wh{\varphi}(\uis{ij}{\gamma}\uih{kl}{\gamma})(\uis{kl}{\gamma}\otimes 1)\\
       &=\sum_{k,l=1}^{n_{\gamma}}\lambda^\gamma_kd_{\gamma}\al(a)\delta_{ik}\delta_{jl}\frac{1}{\lambda_{k}d_{\gamma}}(\uis{kl}{\gamma}\otimes 1)\\
       &=\al(a)(\uis{ij}{\gamma}\otimes 1).
   \end{align*}
   From equation \eqref{eq14}, linearity and normality
   \begin{equation}\label{eq15}
       (\wh{\tl}\op(\chi_{\gamma})\otimes \id)(T)=\sum_{k,l=1}^{n_{\gamma}}\lambda^\gamma_kd_{\gamma}E_{\wh{\wtal}}\left( (\wh{\tl}\op(\chi_{\gamma})\otimes \id)(T)(\uih{kl}{\gamma}\otimes 1)\right)(\uis{kl}{\gamma}\otimes 1)\in\crsred{\G}{A}
   \end{equation}
   for any $T\in \G \leftindex_{\widetilde{\al}} {\ltimes} A^{**}$. Thus, for $\phi \in (\crsred{\G}{A})^*$ and $T\in \G \leftindex_{\widetilde{\al}} {\ltimes} A^{**}$ we have
   \begin{align}\label{eq16}
        \phi\left( \wh{\tl}\op(\chi_{\gamma})\otimes \id)(T)\right)&=\phi\biggl( \sum_{k,l=1}^{n_{\gamma}}\lambda^\gamma_kd_{\gamma} E_{\wh{\wtal}}\left( (\wh{\tl}\op(\chi_{\gamma})\otimes \id)(T)(\uih{kl}{\gamma}\otimes 1)\right)(\uis{kl}{\gamma}\otimes 1)\biggr)\nonumber\\
        &=\sum_{k,l=1}^{n_{\be}}\lambda^\gamma_kd_{\gamma} \phi\lBr \underbrace{E_{\wh{\wtal}}\left( (\wh{\tl}\op(\chi_{\be})\otimes \id)(T)(\uih{kl}{\gamma}\otimes 1)\right)}_{\in \al(A)}(\uis{kl}{\gamma}\otimes 1)\biggr) \rBr
   \end{align}
For $k,l\in \{1,\, \cdots,\, n_{\gamma}\}$, define $\phi_{kl}\in A^*=(A^{**})_*$ by $\phi_{kl}(a):=\phi(\al(a)(\uis{kl}{\gamma}\otimes 1))$. This means we can write \eqref{eq16} in terms of normal functionals $\phi_{kl}$:
      $$\phi\left( \wh{\tl}\op(\chi_{\gamma})\otimes \id)(T)\right)=\sum_{k,l=1}^{n_{\gamma}}\lambda^\gamma_kd_{\gamma} \phi_{kl}(a_{kl}^T).$$
      where $a_{kl}^T:=\al^{-1}\lBr E_{\wh{\wtal}}\left( (\wh{\tl}\op(\chi_{\gamma})\otimes \id)(T)(\uih{kl}{\gamma}\otimes 1)\right)\rBr \in A$. The assignment $\G \leftindex_{\widetilde{\al}} {\ltimes} A^{**}\ni T \mapsto a_{kl}^T\in A^{**}$ is normal, which implies that $(\wh{\tl}\op(\chi_{\gamma})\otimes\id)^*(\phi)$ extends to a normal functional on $\G \leftindex_{\widetilde{\al}} {\ltimes} A^{**}$.
   
\end{proof}

For every $T \in \G \leftindex_\al {\ltimes}_r A$ and $\phi \in (\G \leftindex_\al {\ltimes}_r A)^*$, define the linear functional $$\om_{T,\phi}: \McbQHl\rightarrow \C$$ by
$$\om_{T,\phi}(\wh{v}):=\la (\tll{v}\otimes \id)(T), \, \phi \ra,\;\;\;\;\wh{v}\in \McbQHl.$$
\begin{prop}\label{Prop3.7}
    Let $\G$ be a discrete quantum group and $\al$ be an action of $\G$ on a $C^*$--algebra $A$. Then for every $T \in \G \leftindex_\al {\ltimes}_r A$ and $\phi \in (\G \leftindex_\al {\ltimes}_r A)^*$, the functional $\om_{T, \phi}$ lies in $\QcbQHl$.
\end{prop}
\begin{proof}
    First suppose that $T=\al(a)(\uis{ij}{\be}\otimes 1)$ for some $a\in A$ and $\be \in \Irrd$. If $\wh{v}_i\rightarrow \wh{v}$ in the weak* topology of $M_{cb}(\LOQHOP)$, then $(\wh{\tl}\op(\wh{v_i})\otimes\id)(S)\rightarrow (\wh{\tl}\op(\wh{v})\otimes\id)(S) $ weak* for all $S\in \mc{B}(\ltq)\oten A^{**}$.  Recall that $\chi_{\be}$ is in the center of $\LOQH$, then for any $\wh{u}\in M_{cb}(\LOQHOP)$ and any $\wh{f}\in \LOQHOP$ we have
    \begin{align*}
        m_{\wh{u}}^l(m_{\chi_{\be}}^l(\wh{f}))&=m_{\wh{u}}^l(\chi_{\be}\star \wh{f}) =m_{\wh{u}}^l(\wh{f} \star \chi_{\be})\\
        &=m_{\wh{u}}^l(\wh{f})\star \chi_{\be}=\chi_{\be}\star m_{\wh{u}}^l(\wh{f})=m^l_{\chi_{\be}}(m_{\wh{u}}^l(\wh{f})).
    \end{align*}
    Therefore, for each $\wh{v}_i$, $m_{\wh{v}_i}^l$ and $m_{\chi_{\beta}}^l$ commute, and consequently,  $\tll{v_i}$ and $\wh{\tl}\op(\chi_{\be})$ commute.
    
    By Lemma \ref{lem3.4}, $(\wh{\tl}\op(\chi_{\be})\otimes\id)(T)=T$ and $(\wh{\tl}\op(\chi_{\be})\otimes\id)^*(\phi)\in (\G \leftindex_{\widetilde{\al}} {\ltimes} A^{**})_*$, so then
    \begin{align*}
        \om_{T,\phi}(\wh{v_i})&=\la (\tll{v_i}\otimes \id)(T), \, \phi \ra\\
        &=\la(\tll{v_i}\otimes \id)( (\wh{\tl}\op(\chi_{\be})\otimes \id)T), \, \phi \ra\\
        &=\la(\tll{v_i}\otimes \id)(T), \, (\wh{\tl}\op(\chi_{\be})\otimes \id)^*\phi \ra\\
        &\rightarrow \la(\tll{v}\otimes \id)(T), \, (\wh{\tl}\op(\chi_{\be})\otimes \id)^*\phi \ra\\
        &=\la (\tll{v}\otimes \id)(T), \, \phi \ra\\
        &= \om_{T,\phi}(\wh{v}).
    \end{align*}
    Hence, $\om_{T,\phi}\in \QcbQHl$. By norm density of $\la \al(a)(\uis{ij}{\be}\otimes 1)\,| \; a\in A,\, \be \in\Irrd \ra $ in $\crsred{\G}{A}$, we see that $\om_{T,\phi}\in \QcbQHl$ for any $T\in \crsred{\G}{A}$  as $\omega_{T,\vphi} $ is easily seen to be norm continuous in $T$.
\end{proof}
\begin{cor}\label{cor3.2.7}
    Let $\G$ be a discrete quantum group and A be a $C^*$--algebra. Then for any $T\in C(\wh{\G})\mten A$ and any $\phi \in (C(\wh{\G})\mten A)^*$ we have $\omega_{T,\phi}\in \QcbQHl$.
\end{cor}
\begin{proof}
    Let $\al:A\rightarrow M(c_0(\G)\mten A):a\mapsto1\ten a$ be the trivial action on A then $\crsred{\G}{A}=C(\wh{\G})\mten A$. By Proposition \ref{Prop3.7} we have $\omega_{T,\phi}\in \QcbQHl$.
\end{proof}
\begin{remark}
     Kraus and Ruan in \cite[ Proposition 5.2]{krausRuan1999AP-Kac}, proved that when $\G$ is a discrete Kac algebra and $H$ is a separable Hilbert space, we have $\omega_{T,\phi}\in \QcbQHl$ for any $T\in C(\wh{\G})\mten \mc{B}(H)$ and any $\phi \in (C(\wh{\G})\mten \mc{B}(H))^*$. Therefore, our result in Corollary \ref{cor3.2.7} and Proposition \ref{Prop3.7} extend this result to discrete quantum groups.
 \end{remark}  
\begin{thm}\label{thm3.9}
    Let $\G$ be a discrete quantum group. Then the following are equivalent
    \begin{enumerate}
        \item $\G$ has the AP.
        \item There exist nets $(\wh{f_{\iota}})_{\iota \in I}\subseteq \polh\cdot \wh{\vphi} \subseteq \LOQH$ and $(\mc{F}_{\iota})_{\iota\in I}$,consisting of finite subsets of $\Irrd$, such that for any action $\alpha$ of $\G$ on a $C^*$--algebra $A$, and any $T\in \crsred{\G}{A}$,
	\begin{equation}\label{eq17fejer}
	    T=\lim_{\iota}\sum_{\beta\in \mc{F}_{\iota}}\sum _{i,j,k=1}^{n_\beta}\frac{d_{\beta}}{\lambda^{\beta}_i}\la \wh{f_{\iota}}, \wh{u}_{ji}^{\beta} \ra E(T((\wh{u}_{ki}^{\beta})^*\otimes 1))(\wh{u}_{kj}^{\beta}\otimes 1),
	\end{equation}
 where $E$ is the conditional expectation induced by the dual of $\al$.
    \end{enumerate} 
\end{thm}
\begin{proof}
    The proof remains largely the same as the proof of Theorem \ref{thm3.5}. We just need to show that the convergence is in the norm topology. Suppose $\G$ has the AP, then there is a net $(\wh{f_{\iota}})\subseteq \polh\cdot\wh{\vphi}$ such that $\wh{\tl}\op(\wh{f_{\iota}})\rightarrow \id_{\LIQHOP}$ in stable point-weak* topology. Repeating the proof of Lemma \ref{lem3.3},  for any $T\in \crsred{\G}{A}$ we have
$$\sum_{\beta\in \mc{F}_{\iota}}\sum _{ijk}\frac{d_{\beta}}{\lambda^{\beta}_i}\la \wh{f_{\iota}}, \wh{u}_{ji}^{\beta} \ra E(T((\wh{u}_{ki}^{\beta})^*\otimes 1))(\wh{u}_{kj}^{\beta}\otimes 1)\nonumber
		=(\tlf\ten \id)(T).$$
By Lemma \ref{Prop3.7}, $\om_{T,\phi }\in \QcbQHl$ for any $\phi \in (\crsred{\G}{A})^*$, then 
$$\la (\tlf\ten \id)(T),\, \phi  \ra=\om_{T,\phi }(\wh{f_{\iota}})\rightarrow \om_{T,\phi }(1)=\la T,\, \phi \ra.$$
That is, $(\tlf\ten \id)(T)\rightarrow T$ weakly in $\crsred{\G}{A}$. Taking convex
combinations of the $\wh{f_{\iota}}$, we may assume that the latter weak convergence is relative to the
norm topology. Hence,
$$T=\lim_{\iota}\sum_{\beta\in \mc{F}_{\iota}}\sum _{ijk}\frac{d_{\beta}}{\lambda^{\beta}_i}\la \wh{f_{\iota}}, \wh{u}_{ji}^{\beta} \ra E(T((\wh{u}_{ki}^{\beta})^*\otimes 1))(\wh{u}_{kj}^{\beta}\otimes 1)$$
in the norm topology of $\crsred{\G}{A}$.

Conversely, suppose that there exist nets $(\wh{f_{\iota}})\subseteq \polh\cdot\wh{\vphi} $ and $(\mc{F}_{\iota})$, such that for any action $\alpha$ of $\G$ on a $C^*$--algebra $A$, Equation \eqref{eq17fejer} holds. Let $\G$ act trivially on $\mc{K}(H)$ for a Hilbert space $H$, then $\crsred{\G}{\mc{K}(H)}=C(\wh{\G})\mten \mc{K}(H)$. From the above discussion , for any $T\in C(\wh{\G})\mten \mc{K}(H)$, $(\tlf\ten \id)(T)\rightarrow T$ in the norm topology of $C(\wh{\G})\mten \mc{K}(H)$. From \eqref{eq3}, every functional  $\om \in \QcbQHl$ is in the form of $\om=\Om_{T,\rho}$, for some $T\in C(\wh{\G})\mten \mc{K}(H)$ and $\rho \in \LOQH \pten \Th$. Then
$$\Om_{T,\rho}(\wh{f_{\iota}})=\la \tlf\ten \id)(T)\, \rho \ra\rightarrow \la T,\, \rho \ra=\Om_{T,\rho}(1).$$
This means that $\wh{f_{\iota}} \rightarrow 1$ in the weak* topology of $\McbQHl$. So $\G$ has the AP.
\end{proof}
\begin{prop}\label{prop3.10}
       Let $\G$ be a discrete quantum group and $\al$ be an action of $\G$ on a von Neumann algebra $\n$. Then for every $T \in \crs{\G}{\n}$ and $\phi \in (\crs{\G}{\n})_*$, the functional $\om_{T, \phi}$ lies in $\QcbQHl$.
\end{prop}
\begin{proof}
    Since $\phi$ is normal, for any $\wh{v}\in \McbQHl$ we have $(\tll{v}\ten \id )_*(\phi)\in (\crs{\G}{\n})_*$. The remainder of the proof proceeds in the same manner as the proof of Proposition \ref{Prop3.7}
\end{proof}
\begin{cor}
     Let $\G$ be a discrete quantum group and $\n$ be a von Neumann algebra. Then for any $T\in \LIQH\oten \n$ and any $\phi \in (\LOQH \pten \n_*)$ we have $\omega_{T,\phi}\in \QcbQHl$.
\end{cor}
\begin{remark}
    Anderson-Sackaney pointed out in \cite[Proposition 2.14]{anderson2022ideals}
that the work of Kraus and Ruan \cite[Proposition 5.2]{krausRuan1999AP-Kac} extends directly from discrete Kac algebras to discrete quantum groups. i.e. 
given a discrete quantum group $\G$ and a Hilbert space $H$, for all $T\in \LIQH\oten\mc{B}(H)$ and for all $\phi \in \LOQH\pten \Th$ we have $\omega_{T, \phi}\in \QcbQHl$. The above corollary provides another proof for this fact.
\end{remark}

\begin{remark} In view of the results of \cite{crannNeufang2022Fejer}, it would be natural to study extensions of the results in this section to co-amenable type I locally compact quantum groups. 
\end{remark}

\section{Applications}

\subsection{Structure of $L^{\infty}(\h{\G})$--bimodules in $\bltq$} In a series of papers \cite{anoussis2014ideals-A(G),anoussis2016idealsFourier,Todorovanoussis2018bimodules}, Annousis, Katavolos and Todorov studied the structure of weak*-closed $\LI$-- (resp. $VN(G)$--)bimodules in $\BLT$ which are invariant under the canonical action of $M(G)$ (resp. $M_{cb}A(G)$). They characterized such ``jointly invariant subspaces'' via closed left ideals in $\LO$ (resp. $A(G)$) in a number of cases. Partial generalizations of some of their results were given in \cite{crannNeufang2022Fejer} for locally compact groups with the AP, and it is natural to pursue similar questions in the quantum group setting.

Given a locally compact quantum group $\G$ and a closed left ideal $J\unlhd\LOQ$, inspired by \cite{Todorovanoussis2018bimodules}, we define
$$\mathrm{Bim}(J^{\perp})=\overline{\la\h{x}a\h{y}\mid a\in J^{\perp}, \ \h{x},\h{y}\in L^{\infty}(\h{\G})\ra}^{w^*}$$
and

$$\mathrm{Ran}(J)=\overline{\la\Theta^r(f)_*(\rho)\mid f\in J, \ \rho\in\TCQ\ra}^{\norm{\cdot}_{\TCQ}},$$
where $\Theta^r(f)_*$ is the pre-adjoint of the normal completely bounded map $\Theta^r(f)$. Obviously $\mathrm{Bim}(J^{\perp})$ is a $L^{\infty}(\h{\G})$--bimodule in $\BLTQ$, and since $\Theta^r$ is a $L^{\infty}(\h{\G})$--bimodule map it is easy to show that $\mathrm{Ran}(J)^{\perp}$
is also a $L^{\infty}(\h{\G})$--bimodule in $\BLTQ$.

 For a locally compact group $G$, it was shown in \cite{Todorovanoussis2018bimodules} that for any $J\unlhd L^1(G)$, $\mathrm{Bim}(J^{\perp})\subseteq \mathrm{Ran}(J)^{\perp}$, and in \cite{crannNeufang2022Fejer} that the reverse inclusion holds when $G$ has the AP. We will prove similar results in the quantum group setting.
 The following preparatory results are quantum group analogues of results in \cites{crannNeufang2022Fejer, Todorovanoussis2018bimodules}.
 \begin{lem}\label{lem4.1}
 	The space $\mathrm{Ran}(J)^{\perp}$ is the intersection of the kernels of the maps $\{\Theta^r(f) \mid f \in J \}$. We write this as 
 	$$\mathrm{Ran}(J)^{\perp}=\ker \Theta^r(J).$$
 \end{lem}
\begin{proof}
	Let $x \in \mathrm{Ran}(J)^{\perp}\subseteq \BLTQ$, then for any $f\in J$ and $\rho \in \TCQ$ we have $0=\la\Theta^r(f)_*(\rho), x \ra=\la\rho, \Theta^r(f)(x)   \ra$, which means $x \in \ker\Theta^r(J)$. Similarly, the reverse inclusion holds.
\end{proof}
 \begin{lem}\label{lem4.2}
 	Let $\G$ be a locally compact quantum group. Then for any closed essential left ideal $J\unlhd\LOQ$,  $\mathrm{Ran}(J)^{\perp}\cap \LIQ=J^\perp$.
 \end{lem}
\begin{proof} Recall that essentiality means $\overline{\la\LOQ\star J\ra}=J$.  Let $a\in \mathrm{Ran}(J)^{\perp}\cap\LIQ$. Then for any $g\in\LOQ$ and $f\in J$,  $$\la a, g\star f \ra=\la a, \Theta^r(f)_*( g) \ra=\la\Theta^r(f)(a), g \ra=0.$$ 
By linearity and essentiality, it follows that $a \in J^{\perp}$.
	 Conversely, assume that $a \in J^{\perp}$, and $f \in J$. Since $J$ is a closed left ideal, for any $g\in \LOQ$, $g\star f \in J$, so that $\la\Theta^r(f)(a), g \ra=\la a, \Theta^r(f)_*(g) \ra=\la a, g\star f  \ra=0$. Therefore $a \in \ker\Theta^r(J)=\mathrm{Ran}(J)^{\perp}$.
\end{proof}
\begin{prop}\label{prop4.3}
	Let $\G$ be a locally compact quantum group. Then for any closed essential left ideal $J\unlhd \LOQ$, $$\mathrm{Bim}(J^{\perp})\subseteq \mathrm{Ran}(J)^{\perp}.$$
    Consequently, we have
    $$\mathrm{Bim}(J^{\perp})\cap \LIQ= \mathrm{Ran}(J)^{\perp}\cap\LIQ=J^{\perp}.$$
\end{prop}
\begin{proof}
Since elements of the form $\wh{x}a\wh{y}$ generate $\mathrm{Bim}(J^{\perp})$, and $\Theta^r(f)$ is weak* continuous,  by Lemma \ref{lem4.1} it suffices to show that $\Theta^r(f)(\h{x}a\h{y})=0$, for all $a \in J^{\perp}$, $f\in J$ and $\h{x}, \h{y}\in \LIQH$. This follows readily from Lemma \ref{lem4.2} and the bimodule property of $\Theta^r$.
\end{proof}

\begin{thm}\label{thm3.3.7}
Let $\G$ be a discrete quantum group with the AP. Then for any closed left ideal $J\unlhd \ell^1(\G)$,
	$$\mathrm{Bim}(J^{\perp})=\mathrm{Ran}(J)^{\perp}.$$
\end{thm}
\begin{proof}
	By discreteness, every closed left ideal is essential, so thanks to Proposition \ref{prop4.3}, we only need to prove $\mathrm{Ran}(J)^{\perp}\subseteq \mathrm{Bim}(J^{\perp})$.
    
	First, we identify $\mc{B}(\ltq)\cong \crs{\G}{\liq}$ via the  extended right co-multiplication
	$$\Gam^r:\mc{B}(\ltq)\rightarrow\mc{B}(\ltq)\oten \liq\;\: T\mapsto V(T\ten 1)V^*.$$
	It follows that the dual action $\wh{\al }:\mc{B}(\ltq)(\cong \crs{\G}{\liq})\rightarrow \LIQH\oten \mc{B}(\ltq)$ is the extended left co-multiplication
	$$\wh{\Gam^l}\op:\mc{B}(\ltq) \rightarrow \LIQH \oten \mc{B}(\ltq)\;\; T\mapsto \widetilde{W}^*(1\ten T)\widetilde{W}.$$
	Let $f\in \loq$, $T=\Gam^r(\wh{x}x)$ for some $\wh{x}\in \LIQH$, $x\in \liq$. Then
	\begin{align*}
E((\Theta^r(f)\ten \id)(T))&=E((\Theta^r(f)\ten \id)(\Gam^r(x\wh{x})))\\
&=E([\Gam^r(\Theta^r(f)(x))](\wh{x}\ten 1))\\
&=(\wh{\vphi}\ten\id)[\hat{\al}\left(\Gam^r(\Theta^r(f)(x))(\hat{x}\ten 1)\right)]\\
&=(\wh{\vphi}\ten\id)(1\ten\Gam^r(\Theta^r(f)(x))(\widehat{\Gam}^{\op}(\hat{x})\ten 1) \\
&=\Gam^r(\Theta^r(f)(x))\wh{\vphi}(\hat{x})\\
&=(\Theta^r(f)\ten \id)\left(\Gam^r(x)\wh{\vphi}(\hat{x})\right)\\
&=(\Theta^r(f)\ten \id)(\wh{\vphi}\ten\id)((1\ten\Gam^r(x))(\widehat{\Gam}^{\op}(\hat{x})\ten 1))\\
&=(\Theta^r(f)\ten \id)(\wh{\vphi}\ten\id)\wh{\al}(\Gam^r(x))\wh{\al}(\widehat{\Gam}^{\op}(\hat{x}))\\
&=(\Theta^r(f)\ten \id)(E(T))
 	\end{align*} 
Then by linearity and normality of $E$ and $\Gamma^{r}$, $E$ and $\Theta^r(f)\ten \id$ commute for all $f$ and $T\in \mc{B}(\ltq)$. Passing through the identification $\mc{B}(\ltq)\cong \crs{\G}{\liq}$ via the  extended right co-multiplication $\Gam^r$, we get $E(\Theta^r(f)(T))=\Theta^r(f)(E(T))$ for all $f\in \loq$ and $T\in \mc{B}(\ltq)$. In particular, $E(\mathrm{Ran}(J)^\perp)=E(\ker \Theta^r(J))\subseteq\mathrm{Ran}(J)^{\perp}\cap\ell^\infty(\G)=J^{\perp}$.
	
	Now, let $T\in \mathrm{Ran}(J)^{\perp}$. Since $\mathrm{Ran}(J)^{\perp}$ is a $\LIQH$--bimodule, for any $\beta\in \Irrd$, $T(\uis{ij}{\beta}\ten 1)\in \mathrm{Ran}(J)^{\perp}$, hence $E(T(\uis{ij}{\beta}\ten 1))\in J^{\perp}$.
	The Fej\'{e}r representation for $T$ from Theorem \ref{thm3.5} then implies that $T \in \mathrm{Bim}(J^{\perp})$.
\end{proof}
The following theorem characterizes certain jointly invariant subspaces of $\bltq$ for a discrete quantum group with the AP. Compare with \cite[Theorem 4.3]{anoussis2014ideals-A(G)}, where they showed that any weak*-closed $\LI$--bimodule of $\BLT$ that is invariant under $\mathrm{Ad}\rho_r:T\mapsto \rho_r T\rho_r^*$ takes the form of $\mathrm{Bim}(J^{\perp})$ for a closed ideal $J\unlhd A(G)$.
\begin{thm}\label{thm3.3.9}
Let $\G$ be a discrete quantum group with the AP. Then for any weak*-closed subset $U\subseteq \mc{B}(\ltq)$, the following are equivalent:
	\begin{enumerate}
		\item There exists a closed left ideal $ J\unlhd \ell^1(\G)$ such that
		$U=\mathrm{Bim}(J^{\perp})=\mathrm{Ran}(J)^{\perp}.$\vspace{5pt}
		\item $U$ is a $L^{\infty}(\h{\G})$--bimodule, invariant under $\wh{\tl}\op(M_{cb}^l(\LOQHOP))\cong \cb{\liq}{\LIQHOP}{\BLTQ}.$
	\end{enumerate}
\end{thm}
\begin{proof}
	$(1)\Rightarrow(2)$: First, assume that $U=\mathrm{Bim}(J^{\perp})=\mathrm{Ran}(J)^{\perp}$. We only need to show that $U$ is invariant under $\wh{\tl}\op(M_{cb}^l(\LOQHOP))$. Let $T\in U$, $\wh{m}\in M_{cb}^l(\LOQHOP)$ and $\om=\Theta^r(f)_*(\rho)$ for some $f \in J$ and $\rho \in T(\ltq)$, then by Remark \ref{r:comm},
 \begin{align*}
     \la \tll{m} (T), \;\om\ra  
     &= \la \Theta^r(f)(\tll{m} (T)),\; \rho \ra\\
     &=\la \tll{m}(\Theta^r(f)(T)), \; \rho\ra\\
     &=\la T, \; \Theta^r(f)_*(\tll{m}_*(\rho)) \ra\\
     &=0
 \end{align*}
 since $\Theta^r(f)_*(\tll{m}_*(\rho)) \in \mathrm{Ran}(J)$. Then, by density of $\la\Theta^r(f)_*(\rho)\mid f\in J, \ \rho\in\TCQ\ra$ in $\mathrm{Ran}(J)$,  we must have $\tll{m} (T) \in \mathrm{Ran}(J)^{\perp}=U$.
 
 $(2)\Rightarrow(1)$: Let $\Gam$ be the canonical action of $\G$ on $\liq$, and identify $\mc{B}(\ltq)\cong \crs{\G}{\liq}$.
 Define 
 $$I:=\overline{\{E(T\wh{g})\;|\; T\in U,\; \wh{g}\in \textnormal{Pol}(\wh{\G})\}\star\loq}^{w^*}\subseteq \liq,$$ and set $J:=I_{\perp}$. $J$ is a closed left ideal of $\loq$ since $I$ is weak* closed and $I\star \loq\subseteq I$. It remains to show $U=\mathrm{Bim}(J^{\perp})$.
 
Let $T\in U$. By definition of $J$ (as $I_{\perp}$), Theorem \ref{thm3.5} implies that $T$ is a weak*-limit of elements in $\mathrm{Bim}(J^{\perp})$, implying $U\subseteq \mathrm{Bim}(J^{\perp})$. Towards the reverse inclusion, observe 
$$T\star f=(f\ten \id )W^*(1\ten T)W \in U, \ \ \ T\in U, \ f\in\loq,$$ 
since $U$ is a weak*-closed $\LIQH$--bimodule and $W\in \liq\oten \LIQH$. Moreover, for $\wh{g}\in\mathrm{Pol}(\wh{\G})$, $E(T\wh{g})=\wh{\Theta^{\ell}}\op(\wh{\vphi})(T\wh{g})\in U$, as $U$ is invariant under  $\wh{\Theta^{\ell}}\op(\wh{\vphi})$. It follows that $J^{\perp}=I\subseteq U$, which, in turn, implies $\mathrm{Bim}(J^{\perp})$.
\end{proof}
\begin{remark}
    In the proof of the above theorem, we defined $J^{\perp}$ using the action of $\loq$ on $\{E(T\wh{g})\;|\; T\in U,\; \wh{g}\in \textnormal{Pol}(\wh{\G})\}$ to ensure that $J$ is a left ideal of $\loq$. However, if we assume that $\G$ is also a Kac algebra then using the traciality of $\wh{\varphi}$ it was shown in \cite[p. 17]{moakhar2018amenable} that $E$ is $\Gamma$--equivariant. Therefore, in the Kac algebra setting, we can choose $J^{\perp}$ to be simply $\overline{\{E(T\wh{g})\;|\; T\in U,\; \wh{g}\in \textnormal{Pol}(\wh{\G})\}}^{w^*}$.
\end{remark}
As a corollary of Theorem \ref{thm3.3.9}, we have the following new result in the group setting which is a dual version of \cite[Theorem 4.3]{anoussis2014ideals-A(G)}.
\begin{cor}
    Let $G$ be a discrete group with the approximation property. Then for any weak*-closed subspace $U\subseteq \mc{B}(\lt)$, the following are equivalent:
    \begin{enumerate}
       \item There exists a  closed left ideal $ J\unlhd \lo$ such that
		$U=\mathrm{Bim}(J^{\perp})=\mathrm{Ran}(J)^{\perp}.$
		\item $U$ is a $VN(G)$--bimodule that is invariant under $\wh{\Theta}(M_{cb}A(G))=\cb{\LI}{VN(G)}{\BLT}$.
    \end{enumerate}
\end{cor}

\begin{prop}
    Let $\G$ be a discrete quantum group, and let $J_1$ and $J_2$ be closed left ideals of $\loq$. Then the following hold:
    \begin{enumerate}[label=(\roman*)]
        \item $\bim{J_1}\vee \bim{J_2}=\mathrm{Bim}((J_1\cap J_2)^{\perp})$, where $\vee$ denotes weak*-closed linear span.
    \end{enumerate}
    Moreover, if $\G$ has the AP and $\wh{\G}$ has Ditkin's property, then:
    \begin{enumerate}[resume, label=(\roman*)]
        \item  $\mathrm{Bim}(J_1^{\perp}\cap J_2^{\perp})=\mathrm{Bim}(J_1^{\perp})\cap\mathrm{Bim}(J_2^{\perp})$.
        \item $\mathrm{Ran}(J_1)\cap \mathrm{Ran}(J_2)=\mathrm{Ran}(J_1\cap J_2)$.
    \end{enumerate}
\end{prop}
\begin{proof}
    \begin{enumerate}[label=(\roman*)]
        \item The inclusion $\bim{J_1}\vee \bim{J_2}\subseteq \mathrm{Bim}((J_1\cap J_2)^{\perp})$ is clear.  The reverse inclusion follows from the fact that $(J_1\cap J_2)^{\perp}=\overline{J_1^{\perp}+J_2^{\perp}}^{w^*}$ and the definition of $\bim{J}$.
        \item It is obvious that $\mathrm{Bim}(J_1^{\perp}\cap J_2^{\perp}) \subseteq\bim{J_1}\cap \bim{J_2}$. To show the reverse inclusion,
        let $J:=\overline{J_1+J_2}$. If $T\in \bltq$ annihilates both $\mathrm{Ran}(J_1)$ and $\mathrm{Ran}(J_2)$, then $\la T, \Theta^r(f_i)(\rho) \ra=0 $ for any $\rho \in \mc{T}(\ell^2(\G))$ $f_i\in J_i$ and $i=1,2$. Hence, by linearity and continuity of the map $\Theta^r$, $T$ must annihilate $\mathrm{Ran}(J)$. Therefore,
        \begin{equation}\label{eqRAN}
            \mathrm{Ran}(J_1)^{\perp}\cap \mathrm{Ran}(J_2)^{\perp}\subseteq \mathrm{Ran}(J)^{\perp}.
        \end{equation}
        Suppose that $T\in \bim{J_1}\cap \bim{J_2}$, then by Theorem \ref{thm3.3.7} and equation \eqref{eqRAN}, $T\in \bim{J}$. But $\bim{J}\subseteq \mathrm{Bim}(J_1^{\perp}\cap J_2^{\perp})$, since $J_i\subset J$, $i=1,2$. Hence $T\in \mathrm{Bim}(J_1^{\perp}\cap J_2^{\perp})$ and the proof is complete.
        \item From Theorem \ref{thm3.3.7} and first part of this proposition we have 
        $$\overline{\mathrm{Ran}(J_1)^{\perp}+\mathrm{Ran}(J_2)^{\perp}}^{w^*}=\mathrm{Ran}(J_1\cap J_2)^{\perp}.$$
       Taking pre-annihilators and using the general fact that $(I_1\cap I_2)^{\perp}=\overline{I_1^{\perp}+I_2^{\perp}}$, we have $\mathrm{Ran}(J_1)\cap \mathrm{Ran}(J_2)=\mathrm{Ran}(J_1\cap J_2)$.
    \end{enumerate}
\end{proof}

\begin{cor}
  Let $\G$ be a discrete quantum group with the AP. Let $\mc{L}$ be the set of all weak*-closed subspaces of $\bltq$ which are $\LIQH$--bimodules and invariant under $\wh{\tl}\op(M_{cb}^l(\LOQHOP))$. Then 
  $$\mc{L}=\{\bim{J}:\, J\subseteq \loq \;\;\text{is a closed left ideal}\,\},$$ and $\mc{L}$ is a lattice under the operations of intersection and weak*-closed linear span $\vee$. Moreover,
  \begin{equation*}
      \begin{aligned}
         & \mathrm{Bim}(J_1^{\perp})\cap\mathrm{Bim}(J_2^{\perp})=\bim{(J_1+J_2)},\\
          &\bim{J_1}\vee \bim{J_2}=\mathrm{Bim}((J_1\cap J_2)^{\perp}).
      \end{aligned}
  \end{equation*}
\end{cor}
\begin{remark}
If we remove the AP condition from the above corollary, the family of all weak*-closed jointly invariant subspaces of $\bltq$ remains a lattice under intersection and weak*-closed linear span. However, we may no longer be able to characterize this lattice in terms of closed left ideals of $\loq$.
\end{remark}
\begin{defn}
    Given a subset $\Sigma\subset \loq\cong M_{cb}^r(\loq)$, define
    $$\nul(\Sigma):=\{T\in\liq\,:\; \Theta^r(\sigma)(T)=0, \; \text{for all}\; \sigma\in \Sigma\}.$$
    We call $\nul(\Sigma)$ the null set of the operators on $\liq$. Additionally, define
    $$\loq \star\Sigma:=\overline{\la f\star \sigma : \sigma\in \Sigma,\; f \in \loq \ra}.$$
\end{defn}
 Obviously, $\loq\star\Sigma$ is a closed left ideal of $\loq$.
 \begin{remark}
   Given that $\Theta^r(f)|_{\liq}$ is the adjoint of  $\msig{f}$, we can easily verify that
    \begin{equation}\label{eq3.11}
        \nul(\Sigma)=(\loq\star\Sigma)^{\perp}.
    \end{equation}
    Hence, $\nul(\Sigma)$ is a weak*-closed right $\loq$--submodule of $\liq$.
 \end{remark}
The following proposition demonstrates that it is sufficient to study sets of the form $\nul(J)$ where $J$ is a closed ideal of $\loq$.
\begin{prop}\label{prop3.3.15}
    For any subset $\Sigma$ of $\loq$, $$\nul(\Sigma)=\nul(\loq \star \Sigma).$$
\end{prop}
\begin{proof}
    Since $\loq$ is unital, then $\Sigma\subset\loq \star \Sigma$. Therefore, if $T\in \nul(\loq \star \Sigma)$ then for any $\sigma\in \Sigma$, $\Theta^r(\sigma)(T)=0$. This implies that $\nul(\loq \star \Sigma)\subseteq\nul(\Sigma)$.
 To show the reverse inclusion, let $T\in \nul(\Sigma)$ and $f=g\star \sigma\in \loq \star \Sigma$, for some $\sigma\in \Sigma$. Then
 $$\Theta^r(f)(T)=\Theta^r(\msig{g})(T)=\Theta^r(g\star \sigma)(T)=\Theta^r(g)\circ\Theta^r(\sigma)(T)=0.$$
 Hence, $T\in \nul(\loq \star \Sigma)$.
\end{proof}
\begin{lem}\label{lem3.3.16}
    Let $J$ be a closed left ideal of $\loq$. Then
    $$\nul(J)=J^{\perp}.$$
\end{lem}
\begin{proof}
   Since $\loq$ is unital, for a closed left ideal $J\unlhd \loq$, we have $J\subseteq \loq\star J\subseteq J$. Thus, $J^{\perp}=(\loq\star J)^{\perp}=\nul(J)$  by relation \eqref{eq3.11}. 
\end{proof}
We now pass from $\liq$ to $\bltq$: The algebra $M_{cb}^r(\loq)\cong\loq$ acts on $\bltq$ via $\Theta^r(\sigma)$ for $\sigma\in \loq$.
\begin{defn}
     Given a subset $\Sigma\subset \loq\cong M_{cb}^r(\loq)$, define
     $$\widetilde{\nul}(\Sigma):=\{T\in \bltq:\;\; \Theta^r(\sigma)(T)=0, \;\;\text{for all} \;\;\sigma\in \Sigma\}.$$
     Then, we have $\nul(\Sigma)=\widetilde{\nul}(\Sigma)\cap \liq$.
\end{defn}
The proof of the next proposition follows exactly as in Proposition \ref{prop3.3.15}.
\begin{prop}\label{prop3.3.18}
    If $\Sigma\subset \loq$, then
    $$\widetilde{\nul}(\Sigma)
    =\widetilde{\nul}(\loq \star \Sigma).$$
\end{prop}
\begin{prop}\label{prop3.3.19}
For every closed left ideal of $\loq$, we have $$\widetilde{\nul}(J)=\mathrm{Ran}(J)^{\perp}.$$
\end{prop}
\begin{proof}
    Let $T\in \mathrm{Ran}(J)^{\perp}$, then $\la T, \Theta^r(f)_*(\rho) \ra=0$ for any $f\in J$ and $\rho\in \mc{T}(\ltq)$. Therefore, $\la \Theta^r(f)(T), \rho \ra=0$, hence $T\in \widetilde{\nul}(J)$. Similarly, the other inclusion holds as well.
\end{proof}
\begin{cor}\label{cor3.3.20}
    If $\G$ is a discrete quantum group with the AP, then for any closed left ideal $J$ of $\loq$ we have
    $$\widetilde{\nul}(J)=\bim{J}.$$
\end{cor}
It was shown in \cite[Theorem 2.8]{anoussis2016idealsFourier} that for a locally compact group $G$ and any subset $\Sigma \subseteq M_{cb}A(G)$ we have $\widetilde{\nul}(\Sigma)=\mathrm{Bim}(\nul(\Sigma))$ . The following theorem is an extension of \cite[Theorem 2.8]{anoussis2016idealsFourier} to discrete quantum groups with the AP.
\begin{thm}\label{thm3.3.21}
    Let $\G$ be a discrete quantum group with the AP. Then for any subset $\Sigma\subseteq\loq$,
    $$\widetilde{\nul}(\Sigma)=\mathrm{Bim}(\nul(\Sigma)).$$
\end{thm}
\begin{proof}
    It follows form equation \eqref{eq3.11} that $\mathrm{Bim}(\nul(\Sigma))=\bim{(\loq \star \Sigma)}$. But from Corollary \ref{cor3.3.20}, $\bim{(\loq \star \Sigma)}=\widetilde{\nul}(\loq \star \Sigma)$ , and finally by Proposition \ref{prop3.3.18} we have $\widetilde{\nul}(\loq \star \Sigma)=\widetilde{\nul}(\Sigma)$.
\end{proof}
We finish the subsection with some applications to the theory of harmonic operators over discrete quantum groups, building on the classical theory over locally compact groups (see, e.g., \cites{chu2004harmonicFunctions,neufangRunde2007harmonicOperators,anoussis2016idealsFourier}).
\begin{defn}
    Let $\Sigma\subseteq\loq$. We call  an element $T\in \liq$ a $\Sigma$--harmonic functional if $\Theta^r(\sigma)(T)=T$ for all $\sigma\in \Sigma$. We denote the set of all $\Sigma$--harmonic functionals by $\mc{H}_{\Sigma}$.

   Similarly, we call $T\in \bltq$ a $\Sigma$--harmonic operator if $\Theta^r(\sigma)(T)=T$ for all $\sigma\in \Sigma$, and denote the set of all $\Sigma$--harmonic operators by $\widetilde{\mc{H}}_{\Sigma}$.
\end{defn}
If we let $\Sigma'=\{\sigma-1\,:\;\; \sigma\in \Sigma\}$, the we have
\begin{align*}
    &\mc{H}_{\Sigma}=\{T\in \liq\,:\;\;\Theta^r(\sigma)(T)=T\,\;\;\text{for all}\;\;\sigma\in \Sigma \}=\nul(\Sigma');\\
    &\widetilde{\mc{H}}_{\Sigma}=\{T\in \bltq:\; \Theta^r(\sigma)(T)=T\,\;\;\text{for all}\;\;\sigma\in \Sigma \}=\widetilde{\nul}(\Sigma').
\end{align*}
As an immediate consequence of Theorem \ref{thm3.3.21}, we have
\begin{cor}\label{cor3.3.23}
    For a discrete quantum group $\G$ with the AP, the weak*-closed $\LIQH$--bimodule $\mathrm{Bim}(\mc{H}_{\Sigma})$ generated by $\mc{H}_{\Sigma}$ coincides with $\widetilde{\mc{H}}_{\Sigma}$.
\end{cor}

It was shown in Theorem \ref{thm3.3.9} that weak*-closed subspaces $U\subseteq\bltq$ are jointly invariant if and only if $U=\bim{J}$ for some closed left ideal $J\unlhd\loq$. From Corollary \ref{cor3.3.20} we have $\bim{J}=\widetilde{\nul}(J)$, providing another equivalent description. Moreover, we can replace $J$ with a subset $\Sigma \subseteq \loq$.
\begin{prop}\label{prop3.3.24}
    Suppose $\G$ is a discrete quantum group with the AP. Let $U\subseteq \bltq$ be a weak*-closed subspace. The following are equivalent.
    \begin{enumerate}
        \item $U$ is a $L^{\infty}(\h{\G})$--bimodule, invariant under $\wh{\tl}\op(M_{cb}^l(\LOQHOP))\cong \cb{\liq}{\LIQHOP}{\BLTQ}$;
        \item there exists a closed left ideal $J\unlhd\loq$ such that $U=\widetilde{\nul}(J)$;
        \item there exists a subset $\Sigma\subseteq\loq$ such that $U=\widetilde{\nul}(\Sigma)$.
    \end{enumerate}
\end{prop}
\begin{proof}
    We discussed the implication $(1)\Rightarrow (2)$ before the proposition. $(2)\Rightarrow (3)$ is trivial, and $(3)\Rightarrow (1)$ follows from Corollary \ref{cor3.3.20} and Theorem \ref{thm3.3.9}.
\end{proof}
It was shown in \cite[Remark 2.16]{anoussis2016idealsFourier} that any weak*-closed jointly invariant subspace $U\subseteq \BLT$ corresponds to $\widetilde{\mc{H}}_{\Sigma}$ for some $\Sigma\subseteq M_{cb}A(G)$. Following Corollary \ref{cor3.3.23}, we observe that every weak*-closed jointly invariant subspace $U\subseteq\bltq$ is of the form $U=\widetilde{\mc{H}}_{\Sigma}$ for some $\Sigma\subseteq\loq$.

\subsection{Structure of $C(\wh{\G})$--bimodules in $\kltq$}
In this section we apply our Fej\'{e}r decomposition on the $C^*$-algebraic setting to characterize certain norm closed $C(\wh{\G})$--bimodules of $\kltq$ arising from weak*-closed ideals in $\ell^1(\G)$.

As previously mentioned, for a locally compact quantum group $\G$, $\LOQ$ can be identified with a norm closed two-sided ideal of $M(\G)=C_0(\G)^*$ via the mapping $\LOQ \ni f\mapsto f|_{C_0(\G)}\in C_0(\G)^*$. If $\G$ is discrete, then $\loq$ is unital and  we may identify $c_0(\G)^*$ with $\loq$. 

From this point forward,  we assume that $\G$ is a discrete quantum group.

Let $J \unlhd \loq$ be a weak*-closed left ideal. Then $J_{\perp}=\{x\in c_0(\G): \la x, f\ra=0,\; \text{for all}\; f\in J\}$ is a norm closed right $\loq$--submodule of $c_0(\G)$. Moreover, it was shown in \cite[Theorem 3.1]{amini2017compact} that $\G$ is discrete if and only if $c_0(\G)=\liq\cap \kltq$. Therefore $J_{\perp}\subseteq\kltq$.

Given a discrete quantum group $\G$ and a weak*-closed left ideal $J\unlhd\loq$, define
\begin{equation*}
    \mathrm{Bim}(J_{\perp})=\overline{\la\wh{x}a\wh{y}\,|\; a\in J_{\perp}, \wh{x}, \wh{y}\in C(\wh{\G})\ra}^{\|\cdot\|}
\end{equation*}
and 
\begin{equation*}
    \mathrm{RAN}(J)=\overline{\la\Theta^r(f)_*(\rho)\mid f\in J, \ \rho\in\TCQ\ra}^{w^*}
\end{equation*}
where $\Theta^r(f)_*$ is the pre-adjoint of the normal completely bounded map $\Theta^r(f)$. Obviously $\mathrm{Bim}(J_{\perp})$ is an $C(\wh{\G})$--bimodule in $\kltq$. Since $\Theta^r(v)$ is an $C(\wh{\G})$--bimodule map for any $v\in M_{cb}^r(\loq)$, by duality, $\Theta^r(v)_*$ is an $C(\wh{\G})$--bimodule map on $\TCQ$Then $\mathrm{RAN}(J)_{\perp}$ is also a $C(\wh{\G})$--bimodule in $\kltq$.

Similar to the previous subsection we will show that under the AP we have the equality $\mathrm{Bim}(J_{\perp})=\mathrm{RAN}(J)_{\perp}$.

\begin{lem}
    Let $\G$ be a discrete quantum group. For any weak*-closed left ideal $J\unlhd \loq$,
    $$\mathrm{Bim}(J_{\perp})\subseteq \mathrm{RAN}(J)_{\perp}.$$
\end{lem}
\begin{proof}
    Since $\mathrm{RAN}(J)_{\perp}$ is a norm closed $C(\wh{\G})$--bimodule, it is suffices to show that $J_{\perp}\subseteq \mathrm{RAN}(J)_{\perp}$. Let $x\in J_{\perp} $ and $\Theta^r(f)_*(\rho)\in \mathrm{RAN}(J)$ for some $f\in J$ and $\rho\in \mc{T}(\ltq)$. Let $\pi(\rho)=\rho|_{\liq}$ be the canonical quotient map from $\mc{T}(\ltq)$ onto $\loq$. Then
    \begin{align*}
        \la x, \Theta^r(f)_*(\rho)\ra&=\la\Theta^r(f)(x), \rho \ra=\la \Theta^r(f)(x), \pi(\rho) \ra
        =\la f\star x, \pi(\rho)\ra\\&=\la x, \underbrace{\pi(\rho)\star f}_{\in J}\ra=0. 
    \end{align*}
    By density, for all $\om\in \mathrm{RAN}(J)$, we have $\la x, \om\ra=0$, which implies that $x\in \mathrm{RAN}(J)_{\perp}$. 
\end{proof}
\begin{prop}
     Let $\G$ be a discrete quantum group. For any weak*-closed left ideal $J\unlhd \loq$,
     $$c_0(\G)\cap \mathrm{RAN}(J)_{\perp}=J_{\perp}=c_0(\G)\cap\mathrm{Bim}(J_{\perp}).$$
\end{prop}
    \begin{proof}
        The equality $J_{\perp}=c_0(\G)\cap\mathrm{Bim}(J_{\perp})$ is evident once we show that $c_0(\G)\cap \mathrm{RAN}(J)_{\perp}=J_{\perp}$. Since $J_{\perp}\subseteq c_0(\G)$, it suffices to show $c_0(\G)\cap \mathrm{RAN}(J)_{\perp}\subseteq J_{\perp}$. 
        
        Let $x\in c_0(\G)\cap \mathrm{RAN}(J)_{\perp}$. Then for all $f\in J$ and for all  $\rho\in \mc{T}(\ltq)$,
        \begin{align*}
            \la x, \pi(\rho)\star f \ra&=\la f\star x, \pi(\rho) \ra\\
            &=\la \Theta^r(f)(x), \pi(\rho) \ra\\
            &=\la \Theta^r(f)(x), \rho \ra\\
            &=\la x, \Theta^r(f)_*(\rho)\ra\\
            &=0
        \end{align*}
        Since $\pi(\mc{T}(\ltq))\star J =\loq\star J=J $, the above shows that for all $f\in J$, $\la x, f \ra=0$, hence, $x\in J_{\perp}$. 
    \end{proof}
    We are now ready to prove a Theorem analogous to Theorem \ref{thm3.3.7} in the $C^*$--algebraic setting.
\begin{thm}
    Let $\G$ be a discrete quantum group with the AP. Then for any weak*-closed left ideal $J\unlhd\loq$,
    $$\mathrm{Bim}(J_{\perp})= \mathrm{RAN}(J)_{\perp}.$$
\end{thm}
\begin{proof}
    We only need to prove that $\mathrm{RAN}(J)_{\perp}\subseteq \mathrm{Bim}(J_{\perp})$. Consider the canonical action $\Gam: c_0(\G)\to M(c_0(\G)\mten c_0(\G))$ of $\G$ on $c_0(\G)$. By restricting the map $\Gam^r$ to $\overline{\la c_0(\G)C(\wh{\G})\ra}^{\norm{\cdot}} \subseteq\bltq$, we can identify $\crsred{\G}{c_0(\G)}$ with $\overline{\la c_0(\G)C(\wh{\G})\ra}^{\norm{\cdot}}$. However, $\overline{\la c_0(\G)C(\wh{\G})\ra}^{\norm{\cdot}}=\kltq$ by \cite[Corollary 3.6]{hu2013convolution}. Therefore,
    $$\kltq \cong \crsred{\G}{c_0(\G)}.$$
    The rest of the proof follows similarly to the proof of Theorem \ref{thm3.3.7}, with just a few minor adjustments. Using the same argument, one can show that $E$ and $\Theta^r(f)$ commute for all $f\in \loq$. For $T\in \mathrm{RAN}(J)_{\perp}$ we have $T(\uis{ij}{\beta}\ten 1)\in \mathrm{RAN}(J)_{\perp}$. By commutativity of $E$ and $\Theta^r(f)$, $f\in J$, and the fact (\cite[Theorem 4.2]{hu2013convolution}) that $$E(\kltq)=\widehat{\Theta^l}^{\rm op}(\wh{\vphi})(\kltq)=\kltq\cap\ell^\infty(\G)=c_0(\G),$$ we obtain $E(T(\uis{ij}{\beta}\ten 1))\in \mathrm{RAN}(J)_{\perp}\cap c_0(\G)=J_{\perp}\subseteq \mathrm{Bim}(J_{\perp})$. The Fej\'{e}r representation for $T$ from Theorem \ref{thm3.9} implies that T is the norm limit of elements of $\mathrm{Bim}(J_{\perp})$. Since $\mathrm{Bim}(J_{\perp})$ is norm closed, we have $T \in \mathrm{Bim}(J_{\perp})$. Hence $\mathrm{RAN}(J)_{\perp}\subseteq \mathrm{Bim}(J_{\perp})$.
\end{proof}

\begin{thm}\label{thm4.28}
    Let $\G$ be a discrete quantum group with the AP. Then for any norm closed subspace $U\subseteq \kltq$, the following are equivalent:
    \begin{enumerate}
		\item There exists a weak*-closed left ideal $ J\unlhd \ell^1(\G)$ such that
		$U=\mathrm{Bim}(J_{\perp})=\mathrm{RAN}(J)_{\perp}.$
		\item $U$ is a $C(\wh{\G})$--bimodule, invariant under $\wh{\tl}\op(M_{cb}^l(\LOQHOP))\cong \cb{\liq}{\LIQHOP}{\BLTQ}$
	\end{enumerate}
\end{thm}
\begin{proof}
    $(1) \Rightarrow (2)$: exactly the same as the proof of Theorem \ref{thm3.3.9}.\\
    $(2) \Rightarrow (1)$: Let $\Gamma$ be the action of $\G$ on $c_0(\G)$, and identify $\kltq \cong \crsred{\G}{c_0(\G)}$. Let $$I:=\overline{\{E(T\wh{g})| \, T\in U, \wh{g}\in \polh)\}\star \loq}^{\norm{\cdot}}\subseteq c_0(\G),$$ and define $J:=I^{\perp}$. $J$ is a weak*-closed left ideal of $\loq$, since $I$ is a norm closed right $\loq$--submodule of $c_0(\G)$. 
    
    Let $T\in U$. By Theorem \ref{thm3.9}, $T$ is a norm limit of elements in $\mathrm{Bim}(J_{\perp})=\mathrm{Bim}(I)$, implying that $U\subseteq \mathrm{Bim}(J_{\perp})$. Conversely, for any $T\in U$, $E(T)=\wh{\Theta^{\ell}}\op(\wh{\vphi})(T)\in U$ by $\wh{\Theta^{\ell}}\op(\wh{\vphi})$--invariance.  Also, for any $f \in \loq$, $E(T)\star f\in U$ since $U$ is a $C(\wh{\G})$--bimodule. Thus, $J_\perp\subseteq U$ and the desired inclusion follows.
\end{proof}
 If $G$ is a compact group, then $G$ has the approximation property. Consequently, as a result of the above discussion, for any weak*-closed left ideal $ J\unlhd A(G)$ we have $\mathrm{Bim}(J_{\perp})=\mathrm{RAN}(J)_{\perp}$.
With Theorem \ref{thm4.28}, we have the following characterization of jointly invariant subspaces of $\KLT$.
\begin{cor}
    Let $G$ be a compact group. Then for any norm closed subspace $U\subseteq \KLT$, the following are equivalent:
    \begin{enumerate}
        \item There exists a weak*-closed left ideal $ J\unlhd A(G)$ such that
		$U=\mathrm{Bim}(J_{\perp})=\mathrm{RAN}(J)_{\perp}.$
		\item $U$ is a $C(G)$--bimodule, invariant under $\Theta^r(M(G))\cong \cb{VN(G)}{\LI}{\BLT}$.
    \end{enumerate}
\end{cor}

\subsection{Fubini crossed products}
The main result of this section is a slice map property for any action of a discrete quantum group with AP, which extends the results in \cite[\S 3]{crannNeufang2022Fejer} from locally compact groups to discrete quantum groups.
Unless otherwise stated, $\G$ refers to a discrete quantum group.
\subsubsection{The von Neumann Algebraic Setting}
\begin{defn}
  Let $\al$ be an action of $\G$ on a von Neumann algebra $\n$. We say that a  weak*-closed subspace $X\subseteq \n$ is $\G$--invariant if $\al(X)\subseteq \liq\oten X$, where $\oten$ denotes the weak*-spatial tensor product of dual operator spaces.
\end{defn}

Following \cite[Definition 3.1]{crannNeufang2022Fejer} we define the notion of Fubini crossed product for an action of $\G$ on a von Neumann algebra.

\begin{defn}
    Let $\al$ be an action of $\G$ on a von Neumann algebra $\n$. For a $\G$--invariant weak*-closed subspace $X\subseteq \n$, we define the Fubini crossed product of $X$ by $\G$ as
    \begin{equation}
        \fcrs{\G}{X}:=\{T\in \crs{\G}{\n}\;:\;\; E(T(\wh{v}\ten 1))\in \al(X)\; \;\text{for all}\;\;\wh{v}\in \polh \}.
    \end{equation}
    Where $E$ is the conditional expectation induced by $\wh{\al}$.
\end{defn}
\begin{remark}
    Normality of $E$ and the fact that $\alpha:X\to\liq\oten X$ is a normal isometry (so its range is weak* closed by the closed range theorem \cite[Theorem VI.1.10]{conway}) imply that the set on the right hand side is weak* closed.
\end{remark}
\begin{prop}\label{prop3.4.4}
      Let $\al$ be an action of $\G$ on a von Neumann algebra $\n$. For any $\G$--invariant weak*-closed subspace $X\subseteq \n$ we have
      \begin{equation*}
          \G\bar{\ltimes}X\subseteq \fcrs{\G}{X},
      \end{equation*}
      where $\G\bar{\ltimes}X=\overline{\la\al(X)(\LIQH\ten 1)\ra}^{w^*}$.
\end{prop}
\begin{proof}
    First take $T=\al(x)(\uih{ij}{\beta}\ten 1)\in \G\bar{\ltimes}X$ for some $\beta\in\Irrd$ and $i,j \in \{1,...,n_{\beta}\}$. Then by (\ref{e:align}), $E(T(\uis{kl}{\gamma}\otimes 1))\in\alpha(X)$ for any $\gamma\in \Irrd$, $k,l\in\{1,...,n_\gamma\}$.
Hence, by linearity and density, $ E(T(\wh{v}\otimes 1))\in \al(x)$, $T\in\G  \bar{\ltimes} X$, $\wh{v}\in \polh$.
\end{proof}
The following proposition provides an equivalent description of the Fubini crossed product in terms of the dual co-action.
\begin{prop}\label{prop3.4.5}
    Suppose $\al$ is an action of $\G$ on a von Neumann algebra $\n$, and $X$ is a $\G$--invariant weak*-closed subspace $X\subseteq \n$. Then we have
    \begin{equation}\label{eq23}
        \fcrs{\G}{X}=\{T \in \crs{\G}{\n}\,:\;\; T\cdot f\in \G\bar{\ltimes}X,\; \;\text{for all}\;\;f\in \LOQH\}.
    \end{equation}
\end{prop}
In preparation of the proof, recall that $T\cdot\wh{f}=(\tll{f}\ten \id )(T)$ and by normality of $(\tll{f}\ten \id)$, the set on the right-hand side of equation \eqref{eq23} forms a weak*-closed subspace of $\crs{\G}{\n}$.
\begin{proof}
Denote the right-hand side of equation \eqref{eq23} by $\mathcal{D}$.
 Assume that $T \in  \fcrs{\G}{X}$. Let $\wh{f}=\uis{ji}{\beta}\cdot \wh{\vphi}\in \LOQH$ for some $\beta \in \Irrd$. By Lemma \ref{lem3.3} we have
\begin{equation}\label{eq3.14}
    T\cdot \wh{f}=\sum_{k=1}^{n_{\beta}}E(T(\uis{ki}{\beta}\ten1))(\uih{kj}{\beta}\ten 1).
\end{equation}
Since $T\in \fcrs{\G}{X}$, by definition, $E(T(\uis{ki}{\beta}\ten1))\in \al(X)$ for all $\beta\in \Irrd$. Therefore, equation \eqref{eq3.14} implies that $T\cdot \wh{f}\in \la\al(X)(\LIQH\ten 1)\ra$, Then linearity and density  (recall that $\overline{\polh\cdot \wh{\vphi}}=\LOQH$) show that $T\cdot \wh{f}\in \overline{\la\al(X)(\LIQH\ten 1)\ra}^{w^*}$ for all $\wh{f}\in \LOQH$. Hence $T\in\mathcal{D} $.
 
 To prove the reverse inclusion, first note that for any $T\in\crs{\G}{\n}$,
 \begin{align}\label{eq3.15}
     E(T(\uih{ij}{\beta}\ten 1))&=(\wh{\vphi}\ten \id \ten \id)\wh{\al}(T(\uih{ij}{\beta}\ten 1))\nonumber\\
     &=(\wh{\vphi}\ten \id \ten \id)[\wh{\al}(T)\wh{\al}(\uih{ij}{\beta}\ten 1)]\nonumber\\
     &=(\wh{\vphi}\ten \id \ten \id)[\wh{\al}(T)(\wh{\Gam}\op(\uih{ij}{\beta})\ten 1)]\nonumber\\
     &=\sum_{m=1}^{n_\beta} (\wh{\vphi}\ten \id \ten \id)[\wh{\al}(T)(\uih{mj}{\beta}\ten \uih{im}{\beta}\ten 1)]\nonumber\\
     &=\sum_{m=1}^{n_\beta} (\uih{mj}{\beta}\cdot\wh{\vphi}\ten \id \ten \id)[\wh{\al}(T)(1\ten \uih{im}{\beta}\ten 1)]\nonumber\\
     &=\sum_{m=1}^{n_\beta}  T\cdot\wh{\vphi}_{mj}^{\beta}\,(\uih{im}{\beta}\ten 1).\nonumber
 \end{align}
 Let $T\in \mc{D}$, then for any $\wh{f}\in \LOQH$ we have $T\cdot \wh{f}\in \G\bar{\ltimes}X$. In particular, $T\cdot\wh{\vphi}_{mj}^{\beta}\in \G\bar{\ltimes}X$ for any $\beta\in \Irrd$, $m,j\in\{1,...,n_{\beta}\}$. By the above calculations,
 $$E(T(\uih{ij}{\beta}\ten 1))=\sum_{m=1}^{n_\beta} \underbrace{T\cdot\wh{\vphi}_{mj}^{\beta}}_{\in \G\bar{\ltimes}X }\,(\uih{im}{\beta}\ten 1)\in \G\bar{\ltimes}X \subseteq \bltq\oten X.$$
 On the other hand, $E(T(\uih{ij}{\beta}\ten 1))\in\al(N)\subseteq\liq\oten\n$. Therefore, $$E(T(\uih{ij}{\beta}\ten 1))\in\liq\oten\n\cap\bltq\oten X.$$
 Since $\bltq$ has Property $S_{\sigma}$ \cite[Theorem 1.9]{Kraus-TheSliceProblem} (see also \cite{kraus1991slice}), it follows that
  $\bltq\oten X=\bltq \ften X$ (the Fubini tensor product of $\bltq$ and $X$). Moreover, we have 
  \begin{align*}(\liq\oten\n)\cap(\bltq\oten X)&=(\liq\ften\n)\cap(\bltq\ften X)\\
  &=(\liq\cap \bltq) \ften (\n \cap X)=\liq\ften X.
  \end{align*}
  Thus, $E(T(\uih{ij}{\beta}\ten 1))\in \liq\ften X$. Let $n \in \n $ be such that $\al(n)=E(T(\uih{ij}{\beta}\ten 1))$. Then, following  the discussion in Remark \ref{rem2.5.3} we have $n=n\leftindex_{\al}{\star}\epsilon=(\epsilon\ten \id )\al(n)\in X$. Hence, $E(T(\uih{ij}{\beta}\ten 1))\in \al(X)$, which implies that $T \in \fcrs{\G}{X}$.
\end{proof}
In the setting of locally compact groups, it was shown in \cite{suzuki2017group} that for a discrete group $G$ with the AP acting on a von Neumann algebra $\n$, the Fubini crossed product is equal to $G\bar{\ltimes}X$ for any $G$--invariant weak*-closed subspace $X$. This result was then extended to any locally compact group with the AP in \cite{crannNeufang2022Fejer}. We now prove an extension to discrete quantum groups with the AP. 
\begin{prop}\label{prop3.4.6}
    Let $\G$ be a discrete quantum group with the AP and $\al$ be an action of $\G$ on a von Neumann algebra $\n$. Then for any $\G$--invariant weak*-closed subspace $X\subset \n$ 
    $$\G\bar{\ltimes}X=\fcrs{\G}{X}.$$
\end{prop}
\begin{proof}
    We only need to show that $\fcrs{\G}{X}\subseteq \G\bar{\ltimes}X$. Let $T\in \fcrs{\G}{X}\subseteq \crs{\G}{\n}$, then $E(T(\wh{v}\ten 1))\in \al(X)$ for all $\wh{v}\in \polh$. By Theorem \ref{thm3.5}
    $$ T=w^*-\lim_{\iota}\sum_{\beta\in \mathcal{F}_\iota}\sum _{i,j,k=1}^{n_\beta}\frac{d_{\beta}}{\lambda^{\beta}_i}\la \wh{f_{\iota}}, \wh{u}_{ji}^{\beta} \ra\underbrace{E(T((\wh{u}_{ki}^{\beta})^*\otimes 1))}_{\in \al(X)}\underbrace{(\wh{u}_{kj}^{\beta}\otimes 1)}_{\in \LIQH\ten 1},$$
for a net $(\widehat{f}_\iota)$ witnessing the AP. Therefore, $T\in \overline{\la\al(X)(\LIQH\ten 1)\ra}^{w^*}=\G\bar{\ltimes}X$.
\end{proof}
Proposition \ref{prop3.4.5} allows us to interpret the result of the above proposition as a ``slice map property'' in relation to the Fubini crossed product.
\begin{defn}
    Let $\al$ be an action of a discrete quantum group $\G$ on a von Neumann algebra $\n$. We say that the triple $(\n,\,\G,\,\al)$ has the slice map property if $\G\bar{\ltimes}X=\fcrs{\G}{X}$ for every $\G$--invariant weak*-closed subspace $X\subseteq \n$.
\end{defn}
\begin{cor}
    For a discrete quantum group $\G$ with the AP, any triple $(\n,\,\G,\,\al)$ has the slice map property.
\end{cor}

\subsubsection{$C^*$--Algebraic Setting}

We have similar notions in the setting of $C^*$--algebraic actions of discrete quantum groups. The proofs of results in the von Neumann algebraic context apply unchanged in the $C^*$--algebraic setting. Therefore, proofs in this subsection are omitted.


\begin{defn}
    Let $\al$ be an action of $\G$ on a $C^*$--algebra $A$. For a $\G$--invariant norm closed subspace $X\subseteq A$, we define the Fubini crossed product of $X$ by $\G$ as
    \begin{equation}
        \fcrs{\G}{X}:=\{T\in \crsred{\G}{A}\;:\;\; E(T(\wh{v}\ten 1))\in \al(X)\; \;\text{for all}\;\;\wh{v}\in \polh \}.
    \end{equation}
\end{defn}
Continuity of $E$ and $\alpha$ imply that $\fcrs{\G}{X}$ is a closed subspace of $\crsred{\G}{A}$.
\begin{prop}\label{prop3.4.10}
      Let $\al$ be an action of $\G$ on a $C^*$--algebra $A$. For any $\G$--invariant closed subspace $X\subseteq A$ we have
      \begin{equation*}
          \G\bar{\ltimes}_rX\subseteq \fcrs{\G}{X},
      \end{equation*}
      where $\G\bar{\ltimes}_rX=\overline{\la\al(X)(C(\wh{\G})\ten 1)\ra}^{\|\cdot\|}$.
\end{prop}
\begin{defn}
    Let $\al$ be an action of a discrete quantum group $\G$ on a $C^*$--algebra $A$. We say that the triple $(A,\,\G,\,\al)$ has the slice map property if $\G\bar{\ltimes}_rX=\fcrs{\G}{X}$ for every $\G$--invariant closed subspace $X\subseteq A$.
\end{defn}
\begin{prop}
    Let $\G$ be a discrete quantum group with the AP and $\al$ be an action of $\G$ on a $C^*$--algebra $A$. Then for any $\G$--invariant  closed subspace $X\subseteq A$ 
    $$\G\bar{\ltimes}_rX=\fcrs{\G}{X}.$$
\end{prop}

\section*{Acknowledgements}

The authors would like to thank the Isaac Newton Institute for Mathematical Sciences, Cambridge, for support and hospitality during the programme Quantum information, quantum groups and operator algebras, where work on this paper was undertaken. This work was supported by EPSRC grant EP/Z000580/1. J. Crann and M. Neufang were partially supported by the NSERC Discovery Grants RGPIN-2023-05133 and RGPIN-2020-06505, respectively.

\bibliographystyle{unsrt}
\bibliography{references}
\vspace{1cm}

\end{spacing}

\end{document}